\documentclass{birkau}

\usepackage{amsmath,amssymb,url}
\usepackage{enumerate}

\usepackage{tikz-cd}
\usepackage{graphicx}
\numberwithin{equation}{section}

\theoremstyle{plain}
\newtheorem{theorem}{Theorem}[section]
\newtheorem{lemma}[theorem]{Lemma}

\theoremstyle{definition}
\newtheorem{definition}[theorem]{Definition}

\newtheorem{remark}[theorem]{Remark}
\newtheorem{example}[theorem]{Example}

\newcommand{\PWOSC}{{\mathsf{PwoSCO}}}
\newcommand{\PWOSCL}{{\mathsf{PwoSCO_\leq}}}
\newcommand{\FSUPH}{{\mbox{$F\kern -.85ex-\kern -.35ex\mathbb{S}$}}}
\newcommand{\FSUPL}{{\mbox{$F\kern -.85ex-\kern -.35ex\mathbb{S}_\leq$}}}
\newcommand{\smFSUPL}{{{F\kern -.35ex-\kern -.08ex\mathbb{S}_\leq}}}
\newcommand{\FINFH}{{\mbox{$F\kern -.85ex-\kern -.35ex\mathbb{M}$}}}
\newcommand{\FINFL}{{\mbox{$F\kern -.85ex-\kern -.35ex\mathbb{M}_\leq$}}}

\newcommand{\EL}{\mathcal{E}_\leq}
\newcommand{\E}{\mathcal{E}}

\newcommand{\op}{^{\mathrm{op}}}
\newcommand{\uv}[1]{``{#1}"}

\newcommand{\pwo}{\mathsf{pwo}}
\newcommand{\pwos}{\mathsf{pwos}}

\begin{document}


\title[Another look]{Another look on tense and related operators}



\author[M.~Botur]{Michal~Botur}
\address{Department of Algebra and Geometry \\
	Faculty of Science\\
Palack\'y University Olomouc\\
17.\ listopadu 12 \\
771 46 Olomouc \\
Czech Republic}
\email{michal.botur@upol.cz}

\corrauthor[J.~Paseka]{Jan~Paseka}
\address{Department of Mathematics and Statistics \\Faculty of Science\\
	Masaryk University\\
	Kotl\'a\v rsk\'a 2 \\
	611 37 Brno \\
	Czech Republic}
\email{paseka@math.muni.cz}

\author[R.~Smolka]{Richard~Smolka}
\address{Department of Mathematics and Statistics \\Faculty of Science\\
	Masaryk University\\
	Kotl\'a\v rsk\'a 2 \\
	611 37 Brno \\
	Czech Republic}
\email{394121@mail.muni.cz}

\thanks{The research of the first author was supported by the IGA under
grant no.~P\v rF~2021 030. The second authors acknowledges the support by the Austrian Science Fund (FWF): project I 4579-N and the Czech Science Foundation (GA\v CR): project 20-09869L, entitled ``The many facets of orthomodularity''.
Support of the research of the third author by 
the project  ``New approaches to aggregation operators in analysis and processing
of data'', Nr.~18-06915S by Czech Grant Agency (GA\v{C}R) is gratefully acknowledged.}


\subjclass{primary 08A72, 06F99; secondary 18B35}

\keywords{Tense operator, Sup-semilattice, Frame, Poset, Nucleus, Prenucleus}

\begin{abstract}
	Motivated by the classical work of Halmos on functional mo\-nadic Boolean algebras we derive 
	three basic sup-semilattice constructions, among other things the so-called powersets and powerset operators. 
	Such constructions are extremely useful and can be found in 
	almost all branches of modern mathematics, including algebra, logic and topology. 
	Our three constructions give rise to four covariant and two contravariant functors and constitute 
	three adjoint situations we illustrate in simple examples. 
\end{abstract}

\maketitle


\section{Introduction}\label{sec:intro} 

The paper considers certain sup-semilattice constructions  which encompasses a range of known constructions, 
among other things the so-called powersets and powerset operators. 
Such constructions are extremely useful and can be found in 
almost all branches of modern mathematics, including algebra, logic, and topology. 

Apart from classical sources of applications  for powersets and powerset operators coming e.g. from 
representation theorems for distributive lattices, Boolean algebras and Boolean algebras with operators 
(see \cite{Stone1,Stone2}, \cite{Priestley} and \cite{Halmos}) we have to mention the work of Zadeh \cite{Zadeh} 
who deﬁned $I^X$ as a new powerset object instead of ${\mathcal P}(X)$ (here $I=[0,1]$ is the 
real unit interval and $X$ a set) and introduced, for every map $f\colon X\to Y$ between sets $X$ and $Y$, 
new powerset operators $f^{\leftarrow}\colon  I^X \to  I^Y$ and $f^{\rightarrow}\colon  I^Y \to  I^X$,
such that for $a \in I^X, b \in I^Y, y \in Y,$
$$
f^{\leftarrow}(a)(y) = \bigvee\{ a(x) \mid f(x)=y, x\in X\}\quad \text{and}\quad
f^{\rightarrow}(b) = b\circ f. 
$$

This approach was expanded and investigated from different angles of view and 
further generalisations followed, see, e.g., \cite{DeMitri,Gerla,Nguyen,Yager}.
For illustrative examples of possible applications see, e.g., the introductory part of the paper 
\cite{Solovyov}. In this work, we concentrate on representation and approximation of 
sup-semilattices with unary operations motivated by the work of Halmos on functional monadic Boolean algebras 
\cite{Halmos}.

First, we derive three basic constructions, namely we construct 

\begin{enumerate}[{(i)}]
	\item  a sup-semilattice ${\mathbf L}^{\mathbf J}$ with a unary operation 
	$F$ (called shortly an {\em $F$-sup-semilattice}) from a  sup-semilattice ${\mathbf L}$
	and a relation ${\mathbf J}$ (called a {\em frame}), 
	
	\item a  sup-semilattice ${\mathbf J}\otimes {\mathbf H}$ from a frame ${\mathbf J}$ and 
	an {$F$-sup-semilattice} ${\mathbf H}$, and  
	
	\item a frame ${\mathbf J}[{\mathbf H}, {\mathbf L}]$ from an {$F$-sup-semilattice} ${\mathbf H}$ 
	and a  sup-semilattice ${\mathbf L}$. 
\end{enumerate}

Second, these constructions give rise to four covariant and two contravariant functors. In other words, let 
$\mathbb S$, $\FSUPL$ and $\mathbb J$ denote the categories of   sup-semilattices, 
$F$-sup-semilattices and frames, respectively. We show that, for arbitrary sup-semilattice ${\mathbf L}$, 
frame ${\mathbf J}$ and {$F$-sup-semilattice} ${\mathbf H}$,

\begin{enumerate}[{(1)}]
	\item ${-}^{\mathbf J}\colon {\mathbb S}\to \FSUPL$, ${-}\otimes {\mathbf H}\colon {\mathbb J}\to \mathbb S$,  
	${\mathbf J}\otimes {-}\colon {\FSUPL}\to \mathbb S$, 
	${\mathbf J}[\mathbf H, {-}]\colon {\mathbb S}\to \mathbb J$ are covariant functors, and 
	
	\item ${\mathbf L}^{-}\colon {\mathbb J}\to \FSUPL$ and 
	${\mathbf J}[{-}, \mathbf L, ]\colon {\FSUPL}\to \mathbb J$ are contravariant functors.
\end{enumerate}

Finally, we obtain three adjoint situations 

\begin{enumerate}[{(I)}]
	\item $(\eta, \varepsilon)\colon (\mathbf  J\otimes -)\dashv (-^{\mathbf  J})\colon %
	\mathbb S \to \FSUPL$, 
	\item $(\varphi, \psi)\colon (-\otimes \mathbf  H)\dashv {\mathbf  J}[\mathbf  H,-])\colon %
	\mathbb S \to \mathbb J$, and 
	\item $(\nu, \mu)\colon {\mathbf  J}[-, \mathbf  L])\dashv {\mathbf  L}^{-}\colon %
	\mathbb J \to {\FSUPL}^{op}$. 
\end{enumerate}

This new approach presented inherits the approach by Halmos. Namely, if we fix a sup-semilattice  
${\mathbf L}$ such that we can guarantee that $\mu_{\mathbf H}$ will be an embedding of 
{$F$-sup-semilattices} for all ${\mathbf H}\in \FSUPL$ then we obtain a variant of the classical statement 
that every monadic Boolean algebra is isomorphic to a functional monadic Boolean algebra 
\cite[Theorem 12]{HalmosP}.

The paper is structured as follows. 
After this introduction, in Section \ref{Preliminaries} we provide some notions and notations that will be used in the paper. 
We also describe in  a detail the factorization process in $\mathbb S$ and $\FSUPL$. 
In Section \ref{constructions} we provide the background on our three basic constructions 
${(-)}^{(-)}\colon {\mathbb S}\times {\mathbb J}\to \FSUPL$, 
${(-)}\otimes {(-)}\colon {\mathbb J}\times \FSUPL\to {\mathbb S}$ and 
${{\mathbf J}[(-)},{(-)}]\colon \FSUPL\times {\mathbb S}\to {\mathbb J}$. 
We present the induced adjoint situations $(\eta, \varepsilon), (\varphi, \psi)$  and 
$(\nu, \mu)$ in Section \ref{situations}. To illustrate 
these adjoint situations we give in Section \ref{approaches} three examples.  
We then conclude in Section \ref{Conclusions}.

For notions and concepts concerned, but not explained, please refer to \cite{Joy of cats},
\cite{KP} and \cite{ZL3}. The reader should be aware
that although we tried to make the paper as much self-contained as possible, the
lack of space still compelled us to leave some details for his/her own perusal. 

\section{Preliminaries}\label{Preliminaries}

In this section we firstly recall several known but useful concepts.

\subsection{Notation}
We use the category-theoretic notation for composition of maps, that is, 
for maps $f\colon A\longrightarrow B$ and $g\colon B\longrightarrow C$
we denote their composition by $g\circ f\colon A\longrightarrow C$, so that
$(g\circ f)(a) = g(f(a))$ for all $a\in A$. The set of all
maps from the $A$ to $B$ we denote by the usual $B^A$. For a map
$f\colon A\longrightarrow B$ and a set $I$ we write
$f^I\colon A^I\to B^I$ for the map defined by $f^I(x)(i)=f(x(i))$.

As usual, an order-preserving mapping on a poset is called an \textit{operator}.  We say that it is a 
\textit{preclosure operator} if it is increasing. A preclosure operator is called 
a \textit{closure operator} if it is idempotent. 

A poset $(A,\leq_A)$ equipped with an operator $F_A$
on $A$ is abbreviated by a $\pwo$, denoted by $(A,\leqslant_A,F_A)$. 
Note that a $\pwo$ is an example of an ordered algebra \cite{ZL3}.

A \textit{homomorphism} of $\pwos$ $\mathcal {A}=(A,\leq_A,F_A)$ and
$\mathcal {B}=(B,\leq_B,F_B)$ is an order-preserving mapping $f\colon A\rightarrow B$
which satisfies $F_B(f(a))= f(F_A(a))$ for any $a\in A$. We denote by $\PWOSC$ the
category of $\pwos$ with homomorphisms.

A monotone mapping $f:A\rightarrow B$ is called a \textit{lax morphism} if
\[
F_B(f(a))\leqslant f(F_A(a))
\]
for any $a\in A$. The category of $\pwos$ with lax morphisms as morphisms is denoted
by $\PWOSCL$. Clearly every homomorphism is a lax morphism, so $\PWOSC$ is a subcategory
of $\PWOSCL$ which is not necessarily full.

Recall that an {\em order-embedding} from a poset $(A,\leq_A)$ to a poset
$(B,\leq_B)$ is a mapping $h\colon A\rightarrow B$ such that $a\leq_A a^{'}$ iff
$h(a)\leq_B h(a^{'})$, for all $a,a^{'}\in A$. Every order embedding is necessarily
an injective mapping. We denote by $\E$ the class  
of order-embeddings that are homomorphisms. We also denote by 
$\EL$ the class of lax morphisms 
$h\colon (A,\leqslant_A,F_A)\rightarrow(B,\leqslant_B,F_B)$
of $\pwos$ that are order-embeddings which satisfy the following condition: 
$\forall a,a'\in A,$
\begin{equation}\label{eq5}
	F_B(h(a))\leqslant f(a')\Longrightarrow F_A(a)\leqslant a'.
\end{equation}

Evidently, $\EL$ is closed under composition of morphisms and $\E\subseteq \EL$.

\begin{definition} \em  We denote by $\mathbb{S}$ be a category where the objects are sup-semilattices 
	${\mathbf G}=(G,\bigvee)$ (posets which have all joins) and morphisms  are mappings between 
	them preserving arbirary joins.
\end{definition}

$\mathbb{S}$ is given by an algebraic theory, so it is an equationally presented category. It is 
a monadic category (over the category of sets), because it has free objects.  $\mathbb{S}$ is also 
a monoidal category;  even a star-autonomous category. The dual of a sup-semilattice is its opposite poset, 
which is also a sup-semilattice since every sup-semilattice has also arbitrary meets. 
$\mathbb{S}$ is  both a complete and cocomplete category. The categorical product $\Pi_{i \in I} G_i$ coincides with 
both the cartesian product and the categorical sum. 
For a detailed account on  $\mathbb{S}$  see \cite{joyal}.

\begin{definition}  \em 
	We say that a pair $({\mathbf G},F)$ is an {\em $F$-sup-semilattice} if ${\mathbf G}$ is a  sup-semilattice and $F$ is 
	a map $F\colon{}G\longrightarrow\ G$ satisfying $F(\bigvee X)=\bigvee F(X)$ for any $X\subseteq G$. 
	A {\em homorphism of $F$-semi\-lattices} is a morphism $f\colon {\mathbf G}_1\longrightarrow\ {\mathbf G}_2$  in  $\mathbb{S}$ 
	which satisfies $F_2(f(x))=f(F_1(x))$ for any $x\in G_1$.
\end{definition}

Note that any $F$-sup-semilattice is a $\pwo$ and we may identify $({\mathbf G},\mbox{\rm id}_{G})$ 
with ${\mathbf G}$ for every sup-semilattice ${\mathbf G}$. A {\em lax morphism of $F$-sup-semi\-lattices} is 
a morphism $f:\mathbf G_1\longrightarrow\ \mathbf G_2$  in  $\mathbb{S}$ which is also a lax morphism in 
$\PWOSCL$. 

It is transparent that the class of all  $F$-sup-semilattices and all homomorphisms between them forms a category. 
We denote  this category as $\FSUPH$. Similarly, we represent by $\FSUPL$ 
the category of $F$-sup-semilattices and   lax morphisms between them. 
Then  $\FSUPH$ is a subcategory of $\FSUPL$, 
$\FSUPL$ is a subcategory of $\PWOSCL$ and 
$\FSUPH$ is a subcategory of $\PWOSC$.

Recall also that $F$-sup-semilattices  
are sup-algebras (see \cite{ZL3}).

\begin{remark}\label{modules}
	For any sup-semilattice $\mathbf G$, we will denote by ${\mathcal Q}(\mathbf G)$ the
	sup-semilattice of all
	sup-semilattice endomorphisms $F\colon \mathbf G\to \mathbf G$
	(with the pointwise ordering $F_1\leq F_2$ iff
	$F_1(x)\leq F_2(x)$ for all $x\in G$). Hence there is a one-to-one correspondence (for a fixed $\mathbf G$) between 
	elements of ${\mathcal Q}(\mathbf G)$ and $F$-sup-semilattices over  $\mathbf G$. 
	
	The sup-semilattice ${\mathcal Q}(\mathbf G)$ can be described by means
	of a tensor product of sup-semilattices
	$${\mathcal Q}(\mathbf G)\cong (\mathbf G\otimes \mathbf G{\op}){\op}$$
	\noindent
	(see \cite{joyal}). This isomorphism is given by
	
	$$F\mapsto\bigvee_{F(t)\leq s}t\otimes s.$$
\end{remark}

Similarly as for sup-semilattices, we can introduce the category $\mathbb M$ of inf-semilattices and 
infimum preserving mappings between them,  the category 
$\FINFH$ of $F$-inf-semilattices and all homomorphisms between them, and the 
category $\FINFL$ of $F$-inf-semilattices and all lax morphisms between them, respectively. 
Then $\mathbb S^{op}$ and $\mathbb M$, $\FSUPH^{op}$  and $\FINFH$,  and $\FSUPL^{op}$ and $\FINFL$ 
are isomorphic, respectively. The isomorphism is given by the identity on objects 
( since every sup-semilattice has arbitrary meets) and prescribing to a sup-preserving 
mapping its right adjoint on morphisms and operators, respectively. 

\subsection{Quotients in $\mathbb S$ and $\FSUPL$}

\begin{definition}  \em  Let $\mathbf G$ be a sup-semilattice. A { \em  congruence 
		on the  sup-semilaticce}  $\mathbf G$ is 
	an equivalence relation $\theta$ on $G$ satisfying $S=\{(x_i, y_i)\mid i\in I\}\subseteq \theta$ implies  
	$(\bigvee_{i\in I}x_i)\theta (\bigvee_{i\in I}y_i)$. Let us denote the set of all congruences on $\mathbf G$ as 
	$\mathrm{Con}\, \mathbf G$.
\end{definition}

If $G$ is a sup-semilattice and $\theta$ a sup-semilattice congruence on $G$, the factor set
$G/\theta$ is a sup-semilattice again and the projection $\pi\colon G\to  G/\theta$ is therefore a sup-semilattice morphism.
Recall that if $\theta _i\in \mathrm{Con}\, \mathbf G$ for all $i\in I$, then also 
$\bigcap _i \theta _i \in \mathrm{Con}\, \mathbf G$.\\

An { \em  $F$-congruence on the $F$-sup-semilattice}  $({\mathbf G},F)$ is 
a congruence $\theta$ on $\mathbf G$ satisfying $a\theta b$ implies 
$F(a)\theta F(b)$ for all $a, b\in G$. Note that if $F=\mbox{\rm id}_G$ then any congruence  on $\mathbf G$ is 
an  $F$-congruence on  $({\mathbf G},\mbox{\rm id}_G)$.

For a $\pwo$ $\mathcal {A}=(A,\leq_A,F_A)$, a preclosure operator $j$ on
$A$ is called a \textit{prenucleus} if $j$ is a lax morphism (here one has to recall 
Banaschewski's \cite{ban} theory of prenuclei). We put
$A_j=\{a\in A\mid j(a)=a\}$, $\leq_{A_j}=\leq_A\cap\, \left(A_j\times A_j\right)$ and
$F_{A_j}=j\circ F_A$. A prenucleus is said to be a \textit{nucleus} if it is a closure operator.

\begin{lemma}\label{lemnucleus}
	Let $\mathcal {A}=(A,\leq_A,F_A)$ be  a $\pwo$ and $j$ a nucleus on $\mathcal {A}$.
	Then $\mathcal {A}_j=(A_j,\leq_{A_j},F_{A_j})$ is  a $\pwo$, 
	the inclusion map $i_{A_j}\colon {A_j} \to A$ is a lax morphism such that  
	$i_{A_j}\in \EL$ and 
	the surjection $j\colon A\to {A_j}$ is a homomorphisms of $\pwos$. Moreover, 
	\begin{enumerate}
		\item if $j\colon A\to {A}$ is  a homomorphisms of $\pwos$ then $i_{A_j}\colon {A_j} \to A$  
		is  a homomorphisms of $\pwos$, 
		\item if  $\mathcal {A}$
		is a $F$-sup-semilattice then  $\mathcal {A}_j$ is a  $F$-sup-semilattice.
	\end{enumerate}
\end{lemma}
\begin{proof} Clearly, $\left(A_j,\leq_{A_j}\right)$ is a poset and $F_{A_j}$ is increasing. 
	
	We now verify that  $i_{A_j}\colon {A_j} \to A$  is a lax morphism. Let $a=j(a)\in A_j$. Then 
	$$
	F_A(i_{A_j}(a))=F_A(j(a))\leq j(F_A(a))=i_{A_j}(F_{A_j}(a)).
	$$
	Clearly, $i_{A_j}$ is an order-embedding. 
	Now, let $a,  a'\in A_j\subseteq A$ and assume that $F_A(a)=F_A(i_{A_j}(a))\leq i_{A_j}(a')=a'$. 
	We compute: 
	$$
	F_{A_j}(a)=j(F_A(a))\leq j(a')=a'.
	$$
	
	Let us show that  $j\colon A\to {A_j}$ is a homomorphisms of $\pwos$. Let $a\in A$. Then
	$$
	j(F_A(a))\leqslant j(F_A(j(a)))=F_{A_j}(j(a))\leqslant j(j(F_A(a)))=j(F_A(a)).
	$$
	
	Suppose now that $j\colon A\to {A}$ is  a homomorphisms of $\pwos$ and 
	let $a=j(a)\in A_j$. Then 
	$$
	i_{A_j}(F_{A_j}(a))=j(F_A(a))=F_A(j(a))=F_A(i_{A_j}(a)).
	$$

	Assume that  $\mathcal {A}$ is a $F$-sup-semilattice. Then 
	$\left(A_j,\leqslant_{A_j}\right)$ is
	a sup-semilattice such that $\bigvee_{\mathcal {A}_j}=j\circ \bigvee_{\mathcal {A}}$.
	Let  $M\subseteq A_j$. Then
	$$
	\begin{array}{@{\,}r@{\,}c@{\,}l}
		F_{A_j}\left(\bigvee_{\mathcal {A}_j} M\right)&=&%
		\left(j\circ F_A\circ j\circ \bigvee_{\mathcal {A}}\right)(M)%
		\leq \left(j\circ j\circ F_A\circ \bigvee_{\mathcal {A}}\right)(M)\\
		&=&%
		\left(j\circ F_A\circ \bigvee_{\mathcal {A}}\right)(M)=
		j\left(\bigvee_{\mathcal {A}} \{ F_A(m)\mid {m\in M}\}\right)\\
		&\leq& %
		j\left(\bigvee_{\mathcal {A}} \{ j(F_A(m))\mid {m\in M}\}\right)=%
		\bigvee_{\mathcal {A}_j} \{F_{A_j}(m)\mid {m\in M}\}\\%
		&\leq&F_{A_j}\left(\bigvee_{\mathcal {A}_j} M\right).
	\end{array}
	$$
\end{proof}

Similarly as for sup-semilattices there is a one-to-one correspondence 
between nuclei and $F$-congruences on $F$-sup-semilattices, given by 
$$
\begin{array}{r c l}
	j&\mapsto&\theta_j; \ \text{here}\ %
	\theta_j=\{(a, b)\in G\times G\mid j(a)=j(b)\}\\[0.2cm]
	\theta &\mapsto& j_{\theta};\ \text{here}\ %
	j_{\theta}(x)=\bigvee\{y\in G\mid x\theta y\}.\\
\end{array}
$$
We omit the straightforward verification of the above fact.

\begin{lemma}\label{lemprenucleus}
	Let  $({\mathbf G},F)$ be  $F$-sup-semilattice   and  $j$ a prenucleus on $({\mathbf G},F)$. Then 
	the poset $\left(G_j,\leq_{G_j}\right)$ is a closure system on $\mathbf G$, 
	and the associated closure operator $\mbox{n}(j)$ is given by 
	$$
	\mbox{\rm n}(j)(a)=\bigwedge\{x\in G_j\mid a\leq x\}.
	$$
	Moreover, $\mbox{\rm n}(j)$ is  a nucleus on $({\mathbf G},F)$. 
\end{lemma}
\begin{proof} Since $j$ is order-preserving and increasing we obtain easily that 
	$\left(G_j,\leq_{G_j}\right)$ is a closure system and  $\mbox{n}(j)$ is a closure operator. 
	It remains to verify that $\mbox{n}(j)$ is a lax morphism. Let $a\in G$. We put 
	$$
	E=\{x\in L\mid a\leq x\leq \mbox{n}(j)(a), F(x)\leq \mbox{n}(j)(F(a))\}.
	$$
	Then $a\in E$, $x\in E$ implies 
	$F(j(x))\leq j(F(x))\leq j(\mbox{n}(j)(F(a))=\mbox{n}(j)(F(a))$ and hence 
	$j(x)\in E$. Also, for any non-empty $X\subseteq E$, we have 
	$$
	F(\bigvee X)= \bigvee\{F(x)\mid x\in X\}\leq \mbox{n}(j)(F(a)), 
	$$
	i.e., $\bigvee X\in E$. Hence $t=\bigvee E\in E$ is the largest element of $E$. Then 
	$j(t)\leq t\leq j(t)$ and $t\in G_j$. Therefore $a\leq t\leq \mbox{n}(j)(a)$ implies 
	$t= \mbox{n}(j)(a)\in E$. This says that $F(\mbox{n}(j)(a))\leq \mbox{n}(j)(F(a))$. 
\end{proof}

We will need the following. 

Let  $({\mathbf G},F)$ be  $F$-sup-semilattice   and  $X\subseteq G^{2}$. We put 
$$
\begin{array}{r c l}
	j[X](a)&=&a\vee\bigvee\{c\in G\mid d \leq a, (c,d)\in X\ \text{or}\ (d,c)\in X\}\\
\end{array}
$$

\begin{lemma}\label{lemprenucleus2}
	Let  $({\mathbf G},F)$ be  $F$-sup-semilattice   and  $X\subseteq G^{2}$ such that 
	$(F\times F)(X)\subseteq X$.  Then the mapping 
	$j[X]\colon G \to G$ is a prenucleus on $({\mathbf G},F)$.  Moreover, 
	for every $F$-sup-semilattice   $({\mathbf H},F_{{\mathbf H}})$ and every lax morphism 
	$g\colon G\to H$ of $F$-sup-semilattices such that $g(c)=g(d)$ for all $(c,d)\in X$  there is a unique 
	lax morphism $\overline{g}\colon G_{\mbox{\rm\small n}(j[X])}\to H$ 
	of $F$-sup-semilattices such that $g=\overline{g}\circ \mbox{\rm n}(j[X])$. 
\end{lemma}
\begin{proof} Evidently, $j[X]$ is order-preserving and increasing. Let us show that it is a lax morphism. 
	Let $a\in G$. We compute:
	$$
	\begin{array}{@{}r@{\,\,}c@{\,\,}l}
		F(j[X](a))&=&F(a)\vee\bigvee\{F(c)\in G\mid d \leq a, (c,d)\in X\ \text{or}\ (d,c)\in X\}\\
		&\leq&F(a)\vee\bigvee\{F(c)\in G\mid F(d) \leq F(a), (c,d)\in X\ \text{or}\ (d,c)\in X\}\\
		&\leq&F(a)\vee\bigvee\{u\in G\mid v \leq F(a), (u,v)\in X\ \text{or}\ (v,u)\in X\}\\
		&=&j[X](F(a)).
	\end{array}
	$$
	
	Now, let us put 
	$$
	E=\{x\in G\mid g(x)\leq g(a), a\leq x\}\ \text{and}\ \overline{g}((\mbox{n}(j)(a)))=g(a). 
	$$
	Then $a\in E$. Let $x\in E$. Then $a\leq x\leq j[X](x)$ 
	and  $g(x)\leq g(a)$. We have 
	$$
	\begin{array}{@{}r@{\,\,}c@{\,\,}l}
		g(j[X](x))&=&g(x\vee\bigvee\{c\in G\mid d \leq x, (c,d)\in X\ \text{or}\ (d,c)\in X\})\\
		&=&g(x)\vee\bigvee\{g(c)\in G\mid d \leq x, (c,d)\in X\ \text{or}\ (d,c)\in X\}\\%
		&\leq&g(x)\vee\bigvee\{g(d)\in G\mid g(d) \leq g(x)\}\leq g(x)\leq g(a).\\%
	\end{array}
	$$
	Hence $j[X](x)\in E$. For any non-empty $Y\subseteq E$, we have 
	$$
	g(\bigvee Y)= \bigvee\{g(y)\mid y\in Y\}\leq g(a), 
	$$
	i.e., $\bigvee Y\in E$. Hence $t=\bigvee E\in E$ is the largest element of $E$. Then 
	$j[X)(t)\leq t\leq j[X](t)$ and $t\in G_{j[X]}$. Therefore $a\leq  \mbox{n}(j)(a)\leq t$ implies 
	$g(a)\leq g(\mbox{n}(j)(a))\leq g(t)\leq g(a)$. This says that $\overline{g}$ is correctly 
	defined and $g=\overline{g}\circ \mbox{\rm n}(j[X])$. Let us show that $\overline{g}$ is 
	a  lax morphism of $F$-sup-semilattices. Clearly, $\overline{g}$ is order-preserving. 
	Take $Y\subseteq G_{j[X])}$ and $a\in  G_{j[X])}$. We compute:
	$$
	\overline{g}(\bigvee_{{\mathbf G}_{j[X]}} Y)=g(\bigvee Y)=\bigvee\{g(y)\mid y\in Y\}%
	=\bigvee\{\overline{g}(y)\mid y\in Y\}\leq  \overline{g}(\bigvee_{{\mathbf G}_{j[X]}} Y)
	$$
	and 
	$$
	\begin{array}{@{}r@{\,\,}c@{\,\,}l}
		F_H(\overline{g}(a))&=& F_H({g}(a))\leq g(F(a))=g(F(\mbox{\rm n}(j[X])(a)))\\[0.2cm]
		&\leq& g(\mbox{\rm n}(j[X])(F(a)))=\overline{g}(F_{\mbox{\rm\small n}(j[X])}(a)).
	\end{array}
	$$
\end{proof}


\section{Basic constructions}\label{constructions}

\begin{definition} \em  A {\em frame} $\mathbf J$ is a pair $(T,S)$, where $T$ is a set 
	and $S$ is a binary relation on $T$, i.e., $S\subseteq T^2$. A homomorphism $f$ between 
	frames $\mathbf J_1$  and $\mathbf J_2$  is a map $f:T_1\longrightarrow\ T_2$ such that 
	whenever the pair $(i,j)\in S_1$ then $(f(i),f(j))\in S_2$.\\
	
	Frames with frame homomorphisms obviously form a category, which we denote as  $\mathbb{J}$.
\end{definition}

Note that one of our motivations for this work comes from the modal-logic approach introduced by 
Arthur Prior in 1950's (see \cite{Prior57}). Prior's tense logic was subsequently developed 
by logicians and computer scientists. The relation $S$ is considered  as 
a {\em time preference}, i.e., 
$xSy$  expresses {\uv{$x$ \em  is before $y$}} 
and {\uv{\em $y$ is after $x$}}.

\begin{definition}\label{means} \em Let $\mathbf L=(L,\bigvee)$ be a sup-semilattice 
	and ${\mathbf J}=(T,S)$ a frame. Let us define 
	an $F$-sup-semilattice ${\mathbf L}^{\mathbf J}$ as ${\mathbf L}^{\mathbf J}=(\mathbf L^T,F^{\mathbf J})$, 
	where $$(F^{\mathbf J}(x))(i)=\bigvee \{x(k) \mid {(i,k)\in S}\}$$ 
	for all $x\in L^T$. $F^{\mathbf J}$ will be called an 
	{\em operator on} ${\mathbf L}^{T}$ 
	{\em constructed by means of the frame} ${\mathbf J}$.
\end{definition}

It is evident that Definition \ref{means} is correct. Namely, $F^{\mathbf J}\in {\mathcal Q}(\mathbf L^{T})$ since 
$$\begin{array}{r@{\,\,}c@{\,\,}l}
	(F^{\mathbf J}(\bigvee X))(i)&=&\bigvee \{(\bigvee X)(k) \mid {(i,k)\in S}\}%
	=\bigvee \{\overline{x}(k) \mid {(i,k)\in S}, \overline{x}\in X\}\\[0.2cm]
	&=&\bigvee \{\bigvee \{\overline{x}(k) \mid {(i,k)\in S}\}\mid \overline{x}\in X\}\\[0.2cm]
	&=&\bigvee \{(F^{\mathbf J}(\overline{x}))(i)\mid \overline{x}\in X\}=%
	\left(\bigvee \{F^{\mathbf J}(\overline{x})\mid \overline{x}\in X\}\right)(i)
\end{array}$$ 
for all $i\in T$ and all $X\subseteq L^{T}$. 

Recall that  $F^{\mathbf J}(x)$ can be interpreted as ``a case in future of an element $x$" 
(see \cite[Theorem 6.2]{chapa}). 

\begin{theorem}\label{funSSF} Let $\mathbf L_1$ and $\mathbf  L_2$ be sup-semilattices, 
	let $f\colon \mathbf L_1\longrightarrow\ \mathbf L_2$ be a homomorphism and let $\mathbf J=(T,S)$ be a frame. Then there exists a homomorphism $f^{\mathbf J}:\mathbf L_1^{\mathbf J}\longrightarrow\ \mathbf L_2^{\mathbf J}$ 
	in $\mathbb S_F$ such that, for every $x\in L_1^T$ and every $i\in T$, it holds $(f^{\mathbf J}(x))(i)=f(x(i))$. 
	Moreover, $-^{\mathbf J}$ is a functor from $\mathbb S$ to $\mathbb S_F$.
\end{theorem}
\begin{proof} Since  $\mathbb S$ has arbitrary products we know that $f^{\mathbf J}$ is a morphism in $\mathbb S$. 
	It remains to check that $f^{\mathbf J}$ is an homomorphism. We compute: 
	$$\begin{array}{r@{\,\,}c@{\,\,}l}
		(F^{\mathbf J}(f^{\mathbf J}(x)))(i)&=&\bigvee _{(i,k)\in S} f^{\mathbf J}(x)(k)=\bigvee _{(i,k)\in S} f(x(k))\\[0.2cm]
		&=&f(\bigvee _{(i,k)\in S} x(k))=f(F^{\mathbf J}(x)(i))=(f^{\mathbf J}(F^{\mathbf J}(x)))(i).
	\end{array}$$ 
	The functoriality of $-^{\mathbf J}$ follows from the functoriality of the product construction in $\mathbb S$. 
\end{proof}

\begin{theorem}\label{funJSF}  Let $\mathbf J_1$ and $\mathbf J_2$ be frames, let 
	$t\colon\mathbf  J_1\longrightarrow\ \mathbf J_2$ be a homomorphism of frames 
	and let $\mathbf L$ be a sup-semilattice. Then there exists a lax morphism 
	${\mathbf L}^t\colon {\mathbf L}^{{\mathbf J}_2}\longrightarrow\ {\mathbf L}^{\mathbf J_1}$  
	of $F$-sup-semilattices such that, for every $x\in L^{T_2}$ and every $i\in T_1$, it holds 
	$({\mathbf L}^t(x))(i)=x(t(i))$. Moreover, ${\mathbf L}^{-}$ is a contravariant functor from $\mathbb{J}$ 
	to $\FSUPL$.
\end{theorem}
\begin{proof} Evidently,   ${\mathbf L}^{t}$ is a homomorphism in ${\mathbb S}$. 
	Moreover, we compute:
	$$
	\begin{array}{r c l}
		(F^{{\mathbf J}_1}({\mathbf L}^{t}(x)))(i)&=&%
		\bigvee _{(i,k)\in J_1} {\mathbf L}^{t}(x)(k)=\bigvee _{(i,k)\in S_1} x(t(k))\\[0.2cm]
		&\leq&\bigvee _{(t(i),l)\in S_2} x(l)=(F^{\mathbf J_2}(x))(t(i))=%
		({\mathbf L}^t (F^{\mathbf J_2} (x))(i)
	\end{array}
	$$
	for all $x\in L^{T_2}$ and  $i\in T_1$. Now, let 
	$s\colon\mathbf  J_0\longrightarrow\ \mathbf J_1$ be a homomorphism of frames 
	and let   $x\in L^{T_2}$ and every $i\in T_0$. We compute:
	$$
	{\mathbf L}^{t\circ s}(x)(i)=x(t(s(i))={\mathbf L}^{t}(x)(s(i))={\mathbf L}^{s}({\mathbf L}^{t}(x))(i)=%
	\left(({\mathbf L}^{s}\circ {\mathbf L}^{t})(x)\right)(i).%
	$$
	Clearly, ${\mathbf L}^{{\mathrm id}_{\mathbf J_2}}(x)(i)=x(i)={\mathrm id}_{{\mathbf L}^{\mathbf J_2}}(x)(i)$ 
	for all  $x\in L^{T_2}$ and  $i\in T_2$. Hence ${\mathbf L}^{-}$ is really a contravariant functor from $\mathbb{J}$ 
	to $\FSUPL$.
\end{proof}

Recall that ${\mathbf L}^{t}$ taken as a  homomorphism in ${\mathbb S}$ is nothing else 
as the {\em backward powerset operator} 
$t^{\leftarrow}\colon {\mathbf L}^{{T}_2}\longrightarrow\ {\mathbf L}^{T_1}$ 
first introduced by L. A. Zadeh \cite{Zadeh,Solovyov} for the real unit interval $[0,1]$. 

Let $\mathbf L$ be a sup-semilattice and ${\mathbf J}=(T,S)$ a frame. Then, 
for arbitrary $x\in L$ and $i\in T$, we denote $x_{iS}$ as an element of $L^T$ such that 
$x_{iS}(j)=x$ if and only if $(i,j)\in S$ and otherwise we put $x_{iS}(j)=0$. 
In particular, when $S$ will be the identity relation on $T$ we denote $x_{iS}$ by $x_{i=}$. 

Let $\mathbf H=(\mathbf G, F)$ be an $F$-sup-semilattice  and ${\mathbf J}=(T,S)$ a frame. We put 
$$
[\mathbf J, \mathbf H]=\{(x_{iS}\vee F(x)_{i=}, F(x)_{i=})\mid x\in G, i\in T\}.
$$

This means that we try to encode the unary operation $F$ within an ordinary sup-semilattice ${\mathbf J}\otimes {\mathbf H}$ that will be a quotient of $\mathbf G^{T}$ via 
the prenucleus ${j[\mathbf J, \mathbf H]}$.

\begin{definition} Let $\mathbf J=(T,S)$ be a frame and $\mathbf H=(\mathbf G, F)$ 
	an $F$-sup-semilattice. We then define a sup-semilattice 
	${\mathbf J}\otimes {\mathbf H}$ as follows: 
	$${\mathbf J}\otimes {\mathbf H}={\mathbf G}^T_{j[\mathbf J, \mathbf H]}.$$
\end{definition}

Let ${\mathbf J}_1=(T_1,S_1)$ and ${\mathbf J}_2=(T_2,S_2)$ be frames, ${\mathbf L}$ 
a sup-semilattice and $t\colon T_1\longrightarrow\ T_2$ a map. We 
define a homomorphism $t^{\rightarrow}\colon {\mathbf L}^{T_1}\longrightarrow\ {\mathbf L}^{T_2}$ such that for 
all $x\in L^{T_1}$ and $k\in T_2$ we put 
$(t^{\rightarrow}(x))(k)=\bigvee\{x(i)\mid t(i)=k\}$. In this case $t^{\rightarrow}$ is usually called the 
{\em forward powerset operator} \cite{Zadeh,Solovyov} and $t^{\rightarrow}$ is left adjoint to $t^{\leftarrow}$. Moreover, 
$P\colon \mathbb{J} \to {\mathbb S}$  defined by $P({\mathbf J})= {\mathbf L}^{T}$ and 
$P({t})= t^{\rightarrow}$ is a functor (this follows immediately from \cite[Theorem 2.9]{Rod}).  

\begin{theorem} \label{Jtensor} Let $\mathbf J_1=(T_1,S_1)$ and $\mathbf J_2=(T_2,S_2)$ be frames, 
	$t\colon{}\mathbf J_1\to\ \mathbf J_2$ a homomorphism of frames and $\mathbf H=(\mathbf G, F)$ 
	an $F$-sup-semilattice. Then there exists a unique morphism 
	$t\otimes \mathbf H\colon \mathbf J_1\otimes \mathbf H\to\mathbf  J_2\otimes \mathbf H$ 
	of sup-semilattices such that the following diagram commutes:
	\begin{center}
		
		\begin{tikzpicture}
			\node (alfa') at (7,5) {$\mathbf G^{T_1}$};
			\node (beta') at (7,2) {$\mathbf G^{T_2}$};
			\node (gama') at (12,2) {$\mathbf J_2\otimes \mathbf H$};
			\node (delta') at (12,5) {$\mathbf J_1\otimes \mathbf H$};
			
			\node (1') at (9.5,4.5) {$\mbox{\rm n}(j[\mathbf J_1,\mathbf H])$};
			\node (2') at (9.5,2.5) {$\mbox{\rm n}(j[\mathbf J_2,\mathbf H])$};
			\node (3') at (7.4,3.55) {$t^{\rightarrow}$};
			\node (4') at (11.2,3.55) {$t\otimes \mathbf H$};
			
			\draw [->](alfa') -- (delta');
			\draw [->](beta') -- (gama');
			\draw [->](delta') -- (gama');
			\draw [->](alfa') -- (beta');
		\end{tikzpicture}
	\end{center}
	
	Moreover, then $-\otimes\mathbf  H$ is a functor from $\mathbb{J}$ to ${\mathbb S}$.
\end{theorem}
\begin{proof} Let $x\in G$ and $i\in I$. Then for an arbitrary $l\in J_2$ we get:
	
	$$t^{\rightarrow}(x_{iS_1})(l)=\bigvee \{x_{iS_1}(k) \mid t(k)=l\}\leq x_{t(i)S_2}(l)$$
	
	and
	
	$$t^{\rightarrow}(F_{x(i)})(l)=\bigvee  \{F(x)_{i=}(k) \mid t(k)=l\}=F(x)_{t(i)=}(l).$$
	
	Therefore it holds that 
	
	$$\mbox{\rm n}(j[\mathbf J_2,\mathbf H])(t^{\rightarrow}(x_{iS_1}))\leq %
	\mbox{\rm n}(j[\mathbf J_2,\mathbf H])(x_{t(i)S_2}(l))$$
	
	and
	
	$$\mbox{\rm n}(j[\mathbf J_2,\mathbf H])(t^{\rightarrow}(F_{x(i)}))=%
	\mbox{\rm n}(j[\mathbf J_2,\mathbf H])(F(x)_{t(i)=})=\mbox{\rm n}(j[\mathbf J_2,\mathbf H])(x_{t(i)S_2}(l)).$$

	Hence we obtain that 
	$$\mbox{\rm n}(j[\mathbf J_2,\mathbf H])(t^{\rightarrow}(x_{iS_1}\vee F_{x(i)})))=%
	\mbox{\rm n}(j[\mathbf J_2,\mathbf H])(F(x)_{t(i)=}).$$
	
	From Lemma \ref{lemprenucleus2} we obtain that there is a unique morphism $t\otimes \mathbf H$ 
	of sup-semilattices such that 
	$\mbox{\rm n}(j[\mathbf J_2,\mathbf H])\circ t^{\rightarrow}=(t\otimes \mathbf H)\circ %
	\mbox{\rm n}(j[\mathbf J_1,\mathbf H])$. 
	
	Let us show that $-\otimes\mathbf  H$ is a functor. Evidently, for  
	$\mathbf J=(T,S)$ and $\mbox{\rm id}_T$ we have that 
	$\mbox{\rm id}_T^{\rightarrow}=\mbox{\rm id}_{\mathbf G^{T}}$.
	Hence $\mbox{\rm id}_T\otimes  \mathbf H=\mbox{\rm id}_{\mathbf  J\otimes\mathbf  H}$. 
	Now, let $t\colon{}\mathbf J_1\to\ \mathbf J_2$  and $s\colon{}\mathbf J_2\to\ \mathbf J_3$ 
	be  homomorphisms of frames. Then $(s\circ t)^{\rightarrow}=s^{\rightarrow}\circ t^{\rightarrow}$. From the uniqueness 
	property of $t\otimes\mathbf  H$, $s\otimes\mathbf  H$ and $(s\circ t)\otimes\mathbf  H$, and a little bit of 
	diagram chasing we obtain that $(s\circ t)\otimes\mathbf  H=(s\otimes\mathbf  H)\circ (t\otimes\mathbf  H)$.
\end{proof}

\begin{theorem} \label{Htensor}  Let $\mathbf H_1=(\mathbf G_1, F_1), \mathbf H_2=(\mathbf G_2, F_2)$ be 
	$F$-sup-semilattices,  $f\colon \mathbf H_1 \to \mathbf H_2$ a  lax morphism of $F$-sup-semi\-lattices
	and $\mathbf J=(T,S)$ a frame. Then there is a unique morphism 
	$\mathbf J\otimes f\colon \mathbf J\otimes \mathbf H_1\to\ \mathbf J\otimes \mathbf H_2$ of 
	sup-semilattices such that the following diagram commutes:
	
	\begin{center}
		\begin{tikzpicture}
			
			\node (alfa') at (7,5) {$\mathbf G_1^{T}$};
			\node (beta') at (7,2) {$\mathbf G_2^{T}$};
			\node (gama') at (12,2) {$\mathbf J\otimes \mathbf H_2$};
			\node (delta') at (12,5) {$\mathbf J\otimes \mathbf H_1$};
			
			\node (1') at (9.5,4.5) {$\mbox{\rm n}(j[\mathbf J,\mathbf H_1])$};
			\node (2') at (9.5,2.5) {$\mbox{\rm n}(j[\mathbf J,\mathbf H_2])$};
			\node (3') at (7.45,3.55) {$f^{\mathbf J}$};
			\node (4') at (11.2,3.55) {$\mathbf J\otimes f$};
			
			\draw [->](alfa') -- (delta');
			\draw [->](beta') -- (gama');
			\draw [->](delta') -- (gama');
			\draw [->](alfa') -- (beta');
		\end{tikzpicture}
	\end{center}
	
	Moreover, $\mathbf J\otimes -$ is a functor from $\FSUPL$ to $\mathbb S$.
\end{theorem}
\begin{proof}
	For an arbitrary $x\in G_1$ and $i\in S$, we get:
	$$
	\begin{array}{r@{\,\,}c@{\,\,}l}
		\mbox{\rm n}(j[\mathbf J,\mathbf H_2])(f^{\mathbf J}(x_{iS}))&=&%
		\mbox{\rm n}(j[\mathbf J,\mathbf H_2])(f(x)_{iS})\leq %
		\mbox{\rm n}(j[\mathbf J,\mathbf H_2])(F_2(f(x))_{i=})\\[0.2cm]
		&\leq&\mbox{\rm n}(j[\mathbf J,\mathbf H_2])(f(F_1(x))_{i=})=%
		\mbox{\rm n}(j[\mathbf J,\mathbf H_2])(f^{\mathbf J}(F_1(x)_{i=})).
	\end{array}
	$$
	Hence 
	$$
	\mbox{\rm n}(j[\mathbf J,\mathbf H_2])(f^{\mathbf J}(x_{iS}\vee F_1(x)_{i=}))=%
	\mbox{\rm n}(j[\mathbf J,\mathbf H_2])(f^{\mathbf J}(F_1(x)_{i=})).
	$$
	We conclude as before  from Lemma \ref{lemprenucleus2}  that there is a unique morphism 
	$\mathbf J\otimes f\colon \mathbf J\otimes \mathbf H_1\to\ \mathbf J\otimes \mathbf H_2$
	of sup-semilattices such that 
	$\mbox{\rm n}(j[\mathbf J,\mathbf H_2])\circ f^{\mathbf J}=(\mathbf J\otimes f)\circ %
	\mbox{\rm n}(j[\mathbf J,\mathbf H_1])$.

	Let us verify that $\mathbf J\otimes -$ is a functor. 
	Let $\mathbf H=(\mathbf G, F)$ be an $F$-sup-semilattice and 
	$\mbox{\rm id}_{\mathbf H}$ the identity morphism on $\mathbf H$.

	We know from Theorem  \ref{funSSF} that $-^{\mathbf J}$ is a functor from $\mathbb S$ to $\mathbb S_F$, 
	hence also from  $\mathbb S_F$ to $\mathbb S$ (we take twice the forgetful functor  from  $\mathbb S_F$ to $\mathbb S$).
	Hence  we have that 
	$\mbox{\rm id}_{\mathbf H}^{\mathbf J}=\mbox{\rm id}_{\mathbf G^{T}}$. Therefore 
	$\mathbf J\otimes \mbox{\rm id}_{\mathbf H}=\mbox{\rm id}_{\mathbf J\otimes \mathbf H}$.

	Now, let $f\colon{}\mathbf H_1\to\ \mathbf H_2$  and $g\colon{}\mathbf H_2\to\ \mathbf H_3$ 
	be   lax morphisms of $F$-sup-semi\-lattices. 
	Then $(g\circ f)^{\mathbf J}=g^{\mathbf J}\circ f^{\mathbf J}$. From the uniqueness 
	property of $\mathbf J\otimes  f$, $\mathbf J\otimes g$ and $\mathbf J\otimes (g\circ f)$ we conclude  that 
	$\mathbf J\otimes (g\circ f)=(\mathbf J\otimes g)\circ (\mathbf J\otimes  f)$.
\end{proof}

\begin{definition}\label{framefromfss}  Let $\mathbf H=(\mathbf G, F)$ be an $F$-sup-se\-mi\-lattice and 
	let $\mathbf L$ be a sup-se\-mi\-lattice. We define a frame $\mathbf J[\mathbf H, \mathbf L]=%
	(T_{[\mathbf H, \mathbf L]},S_{[\mathbf H, \mathbf L]})$ 
	such that $T_{[\mathbf H, \mathbf L]}={\mathbb S}(\mathbf G, \mathbf L)$, i.e., 
	the elements of $T_{[\mathbf H, \mathbf L]}$ are morphisms $\alpha$ of sup-semilattices 
	from $\mathbf G$ to $\mathbf L$. The relation $S_{[\mathbf H, \mathbf L]}$ is defined for 
	$\alpha,\beta\in T_{[\mathbf H, \mathbf L]}$ as follows:
	\begin{center}
		$\alpha  \mathrel{S_{[\mathbf H, \mathbf L]}} \beta$ if and only for all $x\in G$  $\beta(x)\leq \alpha(F(x))$.
	\end{center}
\end{definition}

\begin{theorem}\label{funSJrel} Let $\mathbf L_1, \mathbf L_2$ be sup-semi-lattices, let be $\mathbf H=(\mathbf G, F)$ an $F$-semilattice 
	and  let $f\colon \mathbf  L_1\to \ \mathbf L_2$ be a morphism of sup-semi-lattices. Then there exists 
	a homomorphism $\mathbf J[\mathbf H, f]\colon \mathbf J[\mathbf H, \mathbf L_1]\to \mathbf J[\mathbf H, \mathbf L_2]$ 
	of frames such that 
	$$(\mathbf J[\mathbf H, f](\alpha))(x)=f(\alpha(x))$$ 
	for all $\alpha\in T_{[\mathbf H, \mathbf L_1]}$ and for all $x\in G$. Moreover, $\mathbf J[\mathbf H, -]$ 
	is a functor from $\mathbb S$ to $\mathbb{J}$.
\end{theorem}
\begin{proof} Let $\alpha, \beta\in T_{[\mathbf H, \mathbf L_1]}$ be such that $\alpha \mathrel{S_{[\mathbf H, \mathbf L_1]}} \beta$. 
	Then for any $x\in G$ we have that $\beta(x)\leq \alpha(F(x))$. Therefore 
	$(\mathbf J[\mathbf H, f](\beta))(x)=f(\beta(x))\leq f(\alpha(F(x))=(\mathbf J[\mathbf H, f](\alpha))(x)$. We obtain  
	$\mathbf J[\mathbf H, f](\alpha) \mathrel{S_{[\mathbf H, \mathbf L_1]}} [H,f]\mathbf J[\mathbf H, f](\beta)$.
	
	Let us check that  $\mathbf J[\mathbf H, -]$ is a functor. Let $\mathbf L$ be a sup-semi-lattice, 
	$\mbox{\rm id}_{\mathbf L}$ the identity morphism of sup-semilattices on $\mathbf L$ and 
	$\alpha\in T_{[\mathbf H, \mathbf L]}$. We compute:
	$$
	\mathbf J[\mathbf H, \mbox{\rm id}_{\mathbf L}](\alpha)=%
	\mbox{\rm id}_{\mathbf L} \circ \alpha=\alpha=\mbox{\rm id}_{\mathbf J[\mathbf H, \mathbf L]}(\alpha). 
	$$
	
	Now, let $f\colon{}\mathbf L_1\to\ \mathbf L_2$  and $g\colon{}\mathbf L_2\to\ \mathbf L_3$ 
	be   morphisms of sup-semi\-lattices. 
	Then 
	$$
	\begin{array}{r@{\,\,}c@{\,\,}l}
		\mathbf J[\mathbf H, g\circ f](\alpha)&=&(g \circ f)\circ \alpha=g \circ (f\circ \alpha)=%
		g\circ \mathbf J[\mathbf H, f](\alpha)\\[0.2cm]
		&=&J[\mathbf H, g](\mathbf J[\mathbf H, f](\alpha))=%
		(J[\mathbf H, g]\circ \mathbf J[\mathbf H, f])(\alpha).
	\end{array}$$
\end{proof}

\begin{theorem}\label{funSFJrel}  Let $\mathbf H_1=(\mathbf G_1, F_1), %
	\mathbf H_2=(\mathbf G_2, F_2)$ be $F$-semilattices, $\mathbf L$ a sup-semilattice 
	and $f\colon \mathbf H_1\to\ \mathbf H_2$ a lax morphism of  $F$-semilattices. 
	Then there exists a homomorphism 
	$\mathbf J[f, \mathbf L]\colon \mathbf J[\mathbf H_2, \mathbf L]\to \mathbf J[\mathbf H_1, \mathbf L]$ 
	of frames such that 
	$$(\mathbf J[f, \mathbf L](\alpha))(x)=\alpha(f(x))=(\alpha \circ f)(x)$$ 
	for all $\alpha\in T_{[\mathbf H_2, \mathbf L]}$ and for all $x\in G_1$. Moreover, $\mathbf J[-, \mathbf L]$ 
	is a  contravariant functor functor from $\FSUPL$ to $\mathbb{J}$.
\end{theorem}
\begin{proof}
	Let $\alpha, \beta\in T_{[\mathbf H_2, \mathbf L]}$ be such that 
	$\alpha \mathrel{S_{[\mathbf H_2, \mathbf L]}} \beta$. Then for any $x\in G_2$ we have 
	$\beta(x)\leq \alpha(F(x))$. Therefore 
	$$(\mathbf J[f, \mathbf G](\beta))(x)=\beta (f(x))\leq \alpha (F(f(x))%
	\leq \alpha (f(F(x))=(\mathbf J[f, \mathbf G](\alpha))(F(x)).$$ 
	
	We conclude $\mathbf J[f, \mathbf G](\alpha) \mathrel{S_{[\mathbf H_1, \mathbf L]}} \mathbf J[f, \mathbf G](\beta)$.

	Let us verify that  $\mathbf J[ -, \mathbf L]$ is a contravariant functor. Let 
	$\mathbf H=(\mathbf G, F)$ be an  $F$-sup-semi-lattice, 
	$\mbox{\rm id}_{\mathbf H}$ the identity lax morphism of  $F$-sup-semilattices on $\mathbf H$ and 
	$\alpha\in T_{[\mathbf H, \mathbf L]}$. We compute:
	$$
	\mathbf J[\mbox{\rm id}_{\mathbf H}, \mathbf L](\alpha)=%
	\alpha\circ \mbox{\rm id}_{\mathbf L}  =\alpha=\mbox{\rm id}_{\mathbf J[\mathbf H, \mathbf L]}(\alpha). 
	$$
	
	Now, let $f\colon{}\mathbf H_1\to\ \mathbf H_2$  and $g\colon{}\mathbf H_2\to\ \mathbf H_3$ 
	be   lax morphisms of $F$-sup-semi\-lattices. We have:
	$$
	\begin{array}{r@{\,\,}c@{\,\,}l}
		\mathbf J[g\circ f,\mathbf L](\alpha)&=&\alpha\circ (g \circ f)=%
		(\alpha\circ g) \circ f=  \mathbf J[g,\mathbf L](\alpha) \circ f\\[0.2cm]
		&=&\mathbf J[f,\mathbf L](\mathbf J[g,\mathbf L](\alpha))=%
		(\mathbf J[f,\mathbf L]\circ \mathbf J[g,\mathbf L])(\alpha).
	\end{array}$$
\end{proof}

\begin{remark} Let $\mathbf H=(\mathbf G, F)$ be a finite  $F$-sup-semi-lattice, 
	$\mathbf L$ a finite sup-semilattice  and $\mathbf J=(T,S)$ a finite frame. 
	Then evidently $\mathbf J \otimes \mathbf H $ is a finite  sup-semilattice, 
	${\mathbf L}^{\mathbf J}$ is a finite  $F$-sup-semi-lattice, and 
	$\mathbf J[\mathbf H,\mathbf L]$ is a finite frame. 
	
\end{remark}
\section{Three adjoint situations}\label{situations}

In this Section we introduce three induced adjoint situations $(\eta, \varepsilon), (\varphi, \psi)$  and $(\nu, \mu)$. 

\begin{theorem}\label{adjSFSL} Let $\mathbf J=(T,S)$ be a frame. Then:
\begin{enumerate}[\rm(a)]
\item For an arbitrary $F$-sup-semilattice $\mathbf H=(\mathbf G, F)$ there exists a lax morphism 
$\eta _{\mathbf H}\colon \mathbf H\to (\mathbf J\otimes \mathbf H)^{\mathbf J}$ 
of $F$-sup-semilattices defined in such a way that 
$$(\eta _{\mathbf H}(x))(i)=\mbox{\rm n}(j[\mathbf J,\mathbf H])(x_{i=}).$$ 
Moreover, 
$\eta=(\eta _{\mathbf H}\colon {\mathbf H}\to ({\mathbf J}\otimes {\mathbf H})^{\mathbf J})_{{\mathbf H}\in \smFSUPL}$ 
is a natural transformation between the identity functor on  $\smFSUPL$ and the endofunctor 
$ ({\mathbf J}\otimes -)^{\mathbf J}$.
\item For an arbitrary sup-semilattice $\mathbf L$ there exists a unique sup-semilattice morphism 
$\varepsilon_{\mathbf L} \colon \mathbf J\otimes {\mathbf L}^{\mathbf J}\to \mathbf L$ 
 such that the following diagram commutes:

\begin{center}
\begin{tikzpicture}[scale=0.7]

\node (alfa') at (12,0) {$\mathbf L$};
  
   \node (gama') at (0,5) {$(\mathbf L^T)^T$};
   \node (delta') at (12,5) {$\mathbf J\otimes \mathbf L^{\mathbf J}$};
   
   \node (1') at (6,5.5) {$\mbox{\rm n}(j[\mathbf J,\mathbf L^{\mathbf J}])$};
   \node (2') at (11,2.5) {$\varepsilon_{\mathbf L}$};
   \node (3') at (6,2) {$e_{\mathbf L}$};

  \draw [->](gama') -- (alfa');
  \draw [->](delta') -- (alfa');
  \draw [->](gama') -- (delta');
  
 \end{tikzpicture}
 \end{center}
 
\noindent{}where $e_{\mathbf L}\colon(\mathbf L^{T})^{T}\to \mathbf L$ is defined by 
 $e_{\mathbf L}(\bar{x})=\bigvee _{i\in T} (\bar{x}(i))(i)$ for 
 any element  $\bar{x}\in (\mathbf  L^{T})^{T}$. Moreover, 
 $\varepsilon=(\varepsilon_{\mathbf L}\colon  \mathbf J \otimes \mathbf  L^{\mathbf J}\to %
 \mathbf L)_{\mathbf L\in \mathbb S}$ is a natural transformation betweeen the endofunctor 
 $\mathbf J \otimes (-)^{\mathbf J}$ and the identity functor on  $ \mathbb S$.

\item There exists an adjoint situation 
$(\eta, \varepsilon)\colon (\mathbf  J\otimes -)\dashv (-^{\mathbf  J})\colon %
\mathbb S \to \FSUPL$. 
\end{enumerate}
\end{theorem}
\begin{proof}
{(a):} Let us consider an arbitrary $F$-sup-semilattice $\mathbf H=(\mathbf G, F)$.  Evidently, 
$\eta _{\mathbf H}$ preserves arbitrary joins. 
Assume that $x\in G$ and $i\in J$.  We compute: 
$$
\begin{array}{@{}r@{\,\,}c@{\,\,}l}
(F^{\mathbf J} (\eta _{\mathbf H}(x)))(i)&=&\bigvee \{(\eta _{\mathbf H}(x))(k)\mid (i,k)\in S\}\\[0.2cm]
&=&\bigvee \{ \mbox{\rm n}(j[\mathbf J,\mathbf H])(x_{k=})\mid (i,k)\in S\}\\[0.2cm]
&=& \mbox{\rm n}(j[\mathbf J,\mathbf H])(\bigvee \{x_{k=}\mid (i,k)\in S\})=%
\mbox{\rm n}(j[\mathbf J,\mathbf H])(x_{iS})\\[0.2cm]
&\leq&\mbox{\rm n}(j[\mathbf J,\mathbf H])(F(x)_{i=})=(\eta _{\mathbf H}(F(x)))(i).
\end{array}
$$

Therefore $F^{\mathbf J} \circ \eta _{\mathbf H}\leq \eta _{\mathbf H} \circ F$  and hence  
$\eta _{\mathbf H}$ is a lax morphism. Now, let us assume that 
$\mathbf H_1=(\mathbf G_1, F_1)$ and  %
$\mathbf H_2=(\mathbf G_2, F_2)$ are $F$-sup-semilattices, and that 
$f\colon \mathbf H_1\to \mathbf  H_2$ is a  lax morphism of $F$-sup-semilattices. 
We have to show that the following diagram commutes:
\begin{center}
\begin{tikzpicture}

\node (alfa') at (7,5) {$\mathbf H_1$};
  \node (beta') at (7,2) {$(\mathbf J\otimes \mathbf H_1)^J$};
   \node (gama') at (12,2) {$(\mathbf T\otimes \mathbf H_2)^J$};
   \node (delta') at (12,5) {$\mathbf H_2$};
   
   \node (1') at (9.5,4.5) {$f$};
   \node (2') at (9.5,2.5) {$(\mathbf J\otimes f)^{\mathbf J}$};
   \node (3') at (7.5,3.55) {$\eta _{\mathbf H_1}$};
   \node (4') at (11.5,3.55) {$\eta _{\mathbf H_2}$};
   
  \draw [->](alfa') -- (delta');
  \draw [->](beta') -- (gama');
  \draw [->](delta') -- (gama');
  \draw [->](alfa') -- (beta');
 \end{tikzpicture}
 \end{center}
 Assume that $x\in G_1$ and $i\in J$. We compute: 
 
$$
\begin{array}{@{}r@{\,}c@{\,}l}
  (((\mathbf J\otimes f)^{\mathbf J}\circ\eta _{\mathbf H_1})(x))(i)&=&%
  (\mathbf J\otimes f)(\eta _{\mathbf H_1}(i))=(\mathbf J\otimes f)%
  (\mbox{\rm n}(j[\mathbf J,\mathbf H_1])(x_{i=}))\\[0.2cm]
  &=& \mbox{\rm n}(j[\mathbf J,\mathbf H_2])(f^{\mathbf J}(x_{i=}))=%
  \mbox{\rm n}(j[\mathbf J,\mathbf H_2])(f(x)_{i=})\\[0.2cm]
  &=& ((\eta _{\mathbf H_2}\circ f)(x))(i).
  \end{array}
$$

(b): Let $\mathbf L$  be a  sup-semilattice. Assume that $x\in L^T$ and $i\in T$.  We compute: 
 $$
 \begin{array}{@{}r@{\,}c@{\,}l}
 e(x_{iS})&=&\bigvee _{k\in T}(x_{iS}(k))(k)=\bigvee _{iSk}x(k)=(F^{\mathbf J}(x))(i)=%
 \bigvee_{k\in T}F^{\mathbf J}(x)_{i=}(k)(k)\\[0.2cm]
  &=& e(F^{\mathbf J}(x)_{i=}).
  \end{array}$$ 
  
  Moreover, it is transparent that $e_{\mathbf L}$ preserves arbitrary joins. 

By Lemma \ref{lemprenucleus2}   there is a unique morphism 
$\varepsilon_{\mathbf L}\colon \mathbf J\otimes \mathbf L^{\mathbf J}\to\ \mathbf L$
of sup-semilattices such that 
$e=\varepsilon_{\mathbf L}\circ %
\mbox{\rm n}(j[\mathbf J,{\mathbf L}^{\mathbf J}])$.  

Let us now consider a morphism $f\colon\mathbf L_1\to \mathbf L_2$ of sup-semilattices.  
We have to prove that the following diagram commutes:

\begin{center}
\begin{tikzpicture}

\node (alfa') at (7,5) {$\mathbf J\otimes \mathbf L_1^{\mathbf J}$};
  \node (beta') at (7,2) {$\mathbf L_1$};
   \node (gama') at (12,2) {$\mathbf L_2$};
   \node (delta') at (12,5) {$\mathbf J\otimes\mathbf  L_2^{\mathbf J}$};
   
   \node (1') at (9.5,4.5) {$\mathbf J\otimes f^{\mathbf J}$};
   \node (2') at (9.5,2.5) {$f$};
   \node (3') at (7.7,3.55) {$\varepsilon_{\mathbf L_1}$};
   \node (4') at (11.2,3.55) {$\varepsilon_{\mathbf L_2}$};
   
  \draw [->](alfa') -- (delta');
  \draw [->](beta') -- (gama');
  \draw [->](delta') -- (gama');
  \draw [->](alfa') -- (beta');
 \end{tikzpicture}
 \end{center}

Let $\overline{x}\in ({\mathbf L_1}^{T})^{T}$. We compute:
 $$
 \begin{array}{@{}r@{}c@{\,}l}
(\varepsilon_{\mathbf L_2}\circ  (\mathbf J\otimes f^{\mathbf J}))&(\mbox{\rm n}(j[\mathbf J,{\mathbf L_1}^{\mathbf J}])(\overline{x}))&=
\varepsilon _{\mathbf L_2}(\mbox{\rm n}(j[\mathbf J,{\mathbf L_2}^{\mathbf J}]({f^{\mathbf J}}^{\mathbf J}(\overline{x})))=
e_{\mathbf L_2}({f^{\mathbf J}}^{\mathbf J}(\overline{x}))\\[0.2cm]
&\multicolumn{2}{@{}l}{=\bigvee _{i\in T}(({f^{\mathbf J}}^{\mathbf J}(\overline{x}))(i))(i)=%
\bigvee _{i\in T}({f^{\mathbf J}}(\overline{x}(i)))(i)}\\[0.2cm]
&\multicolumn{2}{@{}l}{=\bigvee _{i\in T}{f}(\overline{x}(i)(i))=f(\bigvee _{i\in T}\overline{x}(i)(i))=%
f(e_{\mathbf L_1}(\overline{x}))}\\[0.2cm]
&\multicolumn{2}{@{}l}{=f((\varepsilon_{\mathbf L_1}\circ %
\mbox{\rm n}(j[\mathbf J,{\mathbf L_1}^{\mathbf J}])(\overline{x}))=%
(f\circ \varepsilon_{\mathbf L_1})%
(\mbox{\rm n}(j[\mathbf J,{\mathbf L_1}^{\mathbf J}])(\overline{x})).}
 \end{array}$$ 

(c): Let  $\mathbf H=(\mathbf G, F)$ be an $F$-sup-semilattice and  $\mathbf L$   a  sup-semilattice.
We will prove the commutativity of following diagrams:

\noindent
\begin{tabular}{@{}c c@{}c}
\begin{tikzpicture}[scale=0.35]

\node (alfa') at (12,0) {$\mathbf  J\otimes \mathbf  H$};
  
   \node (gama') at (0,5) {$\mathbf  J\otimes\mathbf  H$};
   \node (delta') at (12,5) {$\mathbf  J\otimes (\mathbf  J\otimes \mathbf  H)^{\mathbf J}$};
   
   \node (1') at (6,5.7) {$\mathbf  J\otimes \eta_{\mathbf  H}$};
   \node (2') at (13.7,2.7) {$\varepsilon_{\mathbf  J\otimes\mathbf  H}$};
   \node (3') at (6,1.5) {$\mbox{\rm id}_{\mathbf  J\otimes \mathbf  H}$};

  \draw [->](gama') -- (alfa');
  \draw [->](delta') -- (alfa');
  \draw [->](gama') -- (delta');
  
 \end{tikzpicture}
 &&
 \begin{tikzpicture}[scale=0.35]

\node (alfa') at (12,0) {${{\mathbf L}^{\mathbf J}}$};
  
   \node (gama') at (0,5) {${{\mathbf L}^{\mathbf J}}$};
   \node (delta') at (12,5) {$(\mathbf  J\otimes {{\mathbf L}^{\mathbf J}})^{\mathbf  J}$};
   
   \node (1') at (6,5.7) {$\eta _{\mathbf  L^{\mathbf  J}}$};
   \node (2') at (13.7,2.5) {$\varepsilon_{{\mathbf L}^{\mathbf J}}$};
   \node (3') at (6,1.5) {$\mbox{\rm id}_{{\mathbf L}^{\mathbf J}}$};

  \draw [->](gama') -- (alfa');
  \draw [->](delta') -- (alfa');
  \draw [->](gama') -- (delta');
  
 \end{tikzpicture}
\end{tabular}

For the first diagram assume $\overline{x}\in \mathbf G^T$.  From Theorem \ref{Htensor}  
we know that the following diagram commutes (when necessary we forget that some morphisms are actually 
from $\FSUPL$ and we work entirely in {$\mathbb S$}):

\begin{center}
\begin{tikzpicture}[scale=0.6]

\node (alfa') at (7,5) {$\mathbf G^{T}$};
  \node (beta') at (7,0) {$((\mathbf J\otimes \mathbf H)^{T})^{T}$};
   \node (gama') at (15,0) {$\mathbf J\otimes (\mathbf J\otimes \mathbf H)^{\mathbf J}$};
   \node (delta') at (15,5) {$\mathbf J\otimes \mathbf H$};
    \node (deltax') at (22,0) {$\mathbf J\otimes \mathbf H$};
      \node (deltaxz') at (22,-5) {$\mathbf J\otimes \mathbf H$};
   
   \node (1') at (10.8,5.5) {$\mbox{\rm n}(j[\mathbf J,\mathbf H])$};
   \node (2') at (10.82,0.5) {$\mbox{\rm n}(j[\mathbf J, (\mathbf J\otimes \mathbf H)^{\mathbf J}])$};
   \node (3') at (7.85,2.5) {$(\eta _{\mathbf H})^{\mathbf J}$};
   \node (4') at (13.9,2.5) {$\mathbf J\otimes \eta _{\mathbf H}$};
      \node (5') at (18.5,0.5) {$\varepsilon_{\mathbf  J\otimes\mathbf  H}$};
       \node (6') at (20.9,-2.5) {$\mbox{\rm id}_{\mathbf  J\otimes \mathbf  H}$};
        \node (7') at (13.9,-2.85) {$e_{\mathbf J\otimes {\mathbf H}}$};
         \node (8') at (19.9,2.5) {$\mbox{\rm id}_{\mathbf  J\otimes \mathbf  H}$};
   
  \draw [->](alfa') -- (delta');
  \draw [->](beta') -- (gama');
  \draw [->](delta') -- (gama');
  \draw [->](alfa') -- (beta');
    \draw [->](gama') -- (deltax');
     \draw [->](deltax') -- (deltaxz');
        \draw [->](beta') -- (deltaxz');
        \draw[densely dashed,->] (delta') --  (deltax');
 \end{tikzpicture}
 \end{center}

\vskip-0.5cm
We compute:

 $$
 \begin{array}{@{}r@{}l@{}l}
\varepsilon_{\mathbf  J\otimes\mathbf  H}((\mathbf  J\otimes \eta_{\mathbf  H})&(\mbox{\rm n}(j[\mathbf J,{\mathbf G}^{\mathbf J}])(\overline{x})))&=%
(\varepsilon_{\mathbf  J\otimes\mathbf  H}\circ %
\mbox{\rm n}(j[\mathbf J, (\mathbf J\otimes \mathbf H)^{\mathbf J}]))%
((\eta _{\mathbf H})^{\mathbf J}(\overline{x}))\\[0.2cm]
&\multicolumn{2}{@{}l}{=%
e_{\mathbf  J\otimes\mathbf  H}((\eta _{\mathbf H})^{\mathbf J}(\overline{x}))%
=\bigvee_{i\in T} ((\eta _{\mathbf H})^{\mathbf J}(\overline{x})(i))(i)}\\[0.2cm]
&\multicolumn{2}{@{}l}{=%
\bigvee_{i\in T} (\eta _{\mathbf H}(\overline{x}(i)))(i)=%
\bigvee_{i\in T} \mbox{\rm n}(j[\mathbf J,\mathbf H])(\overline{x}(i)_{i=})} \\[0.2cm]
&\multicolumn{2}{@{}l}{=%
\mbox{\rm n}(j[\mathbf J,\mathbf H])(\bigvee_{i\in T} \overline{x}(i)_{i=}) =%
\mbox{\rm n}(j[\mathbf J,\mathbf H])(\overline{x})}.
  \end{array}$$ 

Hence $\varepsilon_{\mathbf  J\otimes\mathbf  H}\circ (\mathbf  J\otimes \eta_{\mathbf  H})=%
\mbox{\rm id}_{\mathbf  J\otimes \mathbf  H}$.

To show the commutativity of the second diagram assume that $\overline{x}\in {\mathbf L}^{T}$ 
and $i\in T$. We compute: 
 $$
 \begin{array}{@{}r@{}l@{}l}
(((\varepsilon _{\mathbf  L})^{\mathbf  J} \circ  \eta_{{\mathbf L}^{\mathbf J}})&(\overline{x}))(i)&=%
((\varepsilon _{\mathbf  L})^{\mathbf  J}(\eta_{{\mathbf L}^{\mathbf J}}(\overline{x})))(i)=%
\varepsilon _{\mathbf  L}(\eta_{{\mathbf L}^{\mathbf J}}(\overline{x})(i))\\[0.2cm]%
&\multicolumn{2}{@{}l}{=%
\varepsilon _{\mathbf  L}(\mbox{\rm n}(j[\mathbf J,{\mathbf L}^{\mathbf J}])(\overline{x}_{i=}))
=e_{{\mathbf L}}(\overline{x}_{i=})=\bigvee_{k\in T} \overline{x}_{i=}(k)(k)=\overline{x}(i).}
  \end{array}$$ 
  
  We conclude that $(\varepsilon _{\mathbf  L})^{\mathbf  J} \circ  \eta_{{\mathbf L}^{\mathbf J}}=%
  \mbox{\rm id}_{{\mathbf  L}^{\mathbf  J}}$.

\end{proof}

\begin{theorem}  Let  $\mathbf H=(\mathbf G, F)$ be an $F$-sup-semilattice. Then:
\begin{enumerate}[\rm(a)]
\item For an arbitrary frame $\mathbf J=(T,S)$, there exists a unique homomorphism of frames 
$\varphi_{\mathbf J}\colon \mathbf J\to {\mathbf J}[\mathbf H, \mathbf J\otimes \mathbf H]$ defined for arbitrary $x\in G$ and $i\in T$ in such a way that 
$$(\varphi_{\mathbf J}(i))(x)=\mbox{\rm n}(j[\mathbf J,{\mathbf H}])(x_{i=}).$$ 
Moreover,  $\varphi=(\varphi_{\mathbf J}\colon %
{\mathbf J}\to {\mathbf J}[\mathbf H, \mathbf J\otimes \mathbf H])_{\mathbf J\in {\mathbb J}}$ 
is a natural transformation  between the identity functor on  $\mathbb J$ and the endofunctor  
${\mathbf J}[\mathbf H, -\otimes \mathbf H]$.

\item For an arbitrary sup-semilattice $\mathbf L$ there exists a unique sup-semilattice morphism 
$\psi_{\mathbf L} \colon {\mathbf J}[\mathbf H, \mathbf L]\otimes {\mathbf H}\to \mathbf L$ 
such that the following diagram commutes:

\begin{center}
\begin{tikzpicture}[scale=0.6]

\node (alfa') at (12,0) {$\mathbf L$};
  
   \node (gama') at (0,5) {${\mathbf G}^{{T}_{[\mathbf H, \mathbf L]}}$};
   \node (delta') at (12,5) {$\mathbf {\mathbf J}[\mathbf H, \mathbf L]\otimes \mathbf H$};
   
   \node (1') at (6,4.3) {$\mbox{\rm n}(j[{\mathbf J}[\mathbf H, \mathbf L],\mathbf H])$};
   \node (2') at (11,2.5) {$\psi_{\mathbf L}$};
   \node (3') at (6,2) {$f_{\mathbf L}$};

  \draw [->](gama') -- (alfa');
  \draw [->](delta') -- (alfa');
  \draw [->](gama') -- (delta');
  
 \end{tikzpicture}
 \end{center}
 
\noindent{}where $f_{\mathbf L}\colon {\mathbf G}^{{T}_{[\mathbf H, \mathbf L]}}\to \mathbf L$ 
 is defined by $f_{\mathbf L}({x})=\bigvee _{\alpha\in {{\mathbf J}[\mathbf H, \mathbf L]}} %
 \alpha({x}(\alpha))$ for any ${x}\in {G}^{{T}_{[\mathbf H, \mathbf L]}}$. Moreover, 
 $\psi=(\psi_{\mathbf L} \colon {\mathbf J}[\mathbf H, %
 \mathbf L]\otimes {\mathbf H}\to \mathbf L)_{\mathbf L\in \mathbb S}$ is a natural transformation 
 between the endofunctor ${\mathbf J}[\mathbf H,  -]\otimes {\mathbf H}$ and the identity functor on  $\mathbb S$.

\item There exists an adjoint situation 
$(\varphi, \psi)\colon (-\otimes \mathbf  H)\dashv {\mathbf  J}[\mathbf  H,-])\colon %
\mathbb S \to \mathbb J$. 
\end{enumerate}

\end{theorem}
\begin{proof}
{(a):}  We have to show that our definition is correct. Assume $X\subseteq G$ and $i\in I$. 
We have: 
$$
 \begin{array}{@{}r@{\,}c@{\,}l}
(\varphi_{\mathbf J}(i))\left(\bigvee X\right)&=&%
\mbox{\rm n}(j[\mathbf J,{\mathbf H}])\left(\left(\bigvee X\right)_{i=}\right)=%
\mbox{\rm n}(j[\mathbf J,{\mathbf H}])\left(\bigvee\{x_{i=}\mid x\in X\}\right)\\[0.2cm]
&=&%
\bigvee\{\mbox{\rm n}(j[\mathbf J,{\mathbf H}])\left(x_{i=}\right)\mid x\in X\}=%
\bigvee\{(\varphi_{\mathbf J}(i))\left(x\right)\mid x\in X\}
\end{array}
$$
Hence $\varphi_{\mathbf J}(i)\in {T}_{[\mathbf H, \mathbf J\otimes \mathbf H]}$. 
Now, let $i,k\in T$ such that $i\mathrel{S}k$ and $x\in G$. We compute:
$$\begin{array}{r c l}
(\varphi_{\mathbf J}(k))(x)&=&\mbox{\rm n}(j[\mathbf J,{\mathbf H}])(x_{k=})%
\leq \mbox{\rm n}(j[\mathbf J,{\mathbf H}])(x_{iS})\leq %
\mbox{\rm n}(j[\mathbf J,{\mathbf H}])((F(x))_{i=})\\[0.1cm]
&=&(\varphi_{\mathbf J}(i))(F(x)).
\end{array}$$ 
Therefore $\varphi_{\mathbf J}(i) %
\mathrel{S_{[\mathbf H, \mathbf J\otimes \mathbf H]}}\varphi_{\mathbf J}(k)$.\\

Let  $t\colon{}\mathbf J_1\to \mathbf J_2$ be a homomorphism of frames. We have to show that the following diagram commutes:
\begin{center}
\begin{tikzpicture}
\node (alfa') at (7,5) {$\mathbf J_1$};
  \node (beta') at (7,2) {${\mathbf J}[\mathbf  H,\mathbf J_1\otimes \mathbf H]$};
   \node (gama') at (13,2) {${\mathbf J}[\mathbf H, \mathbf J_2\otimes \mathbf H]$};
   \node (delta') at (13,5) {$\mathbf J_2$};
   
   \node (1') at (9.5,4.5) {$t$};
   \node (2') at (10.05,2.5) {${\mathbf J}[\mathbf H,t\otimes \mathbf H]$};
   \node (3') at (7.5483,3.55) {$\varphi_{\mathbf J_1}$};
   \node (4') at (12.42,3.55) {$\varphi_{\mathbf J_2}$};
   
  \draw [->](alfa') -- (delta');
  \draw [->](beta') -- (gama');
  \draw [->](delta') -- (gama');
  \draw [->](alfa') -- (beta');
 \end{tikzpicture}
\end{center}

Let $i\in J$ and $x\in G$. We obtain 
$$
 \begin{array}{@{}r@{}l@{}l}
(({\mathbf J}[\mathbf H,t\otimes \mathbf H]\circ \varphi_{\mathbf J_1})&(i))(x)&%
=\left({\mathbf J}[\mathbf  H, t\otimes \mathbf H](\varphi_{\mathbf J_1}(i))\right)(x)
\\[0.2cm]
&\multicolumn{2}{@{}l}{=\left((t\otimes \mathbf H)\circ (\varphi_{\mathbf J_1}(i))\right)(x)%
=\left(t\otimes \mathbf H\right)\left((\varphi_{\mathbf J_1}(i))(x)\right)}\\[0.2cm]
&\multicolumn{2}{@{}l}{=%
\left(t\otimes \mathbf H\right)\left(\mbox{\rm n}(j[\mathbf J_1,{\mathbf H}])(x_{i=})\right)%
=\mbox{\rm n}(j[\mathbf J_2,{\mathbf H}])(x_{t(i)=})} \\[0.2cm]
&\multicolumn{2}{@{}l}{=\left(\varphi_{\mathbf J_2}(t(i))\right)(x)
=\left((\varphi_{\mathbf J_2}\circ t)(i)\right)(x).}
\end{array}$$

{(b):}   It is transparent that $f_{\mathbf L}$ preserves arbitrary joins. 


Let $x\in G$ and $\alpha \in {\mathbf J}[\mathbf H,\mathbf L]$ be arbitrary. We compute:
$$
\begin{array}{@{}r@{\,}l@{\,}l}
e_{\mathbf L}(x_{\alpha S_{[\mathbf H,\mathbf L]}})&=&%
\bigvee \{ \beta(x_{\alpha S_{[\mathbf H,\mathbf L]}}(\beta)) \mid {\beta\in T_{[\mathbf H,\mathbf L]}}\}\\[0.2cm]%
&=&\bigvee \{\beta (x)\mid \alpha \mathrel{S_{[\mathbf H,\mathbf L]}} \beta, {\beta\in T_{[\mathbf H,\mathbf L]}}\}%
\leq \alpha (F(x))\\[0.2cm]%
&=&\bigvee \{\beta(F(x)_{\alpha=}(\beta)) \mid {\beta\in T_{[\mathbf H,\mathbf L]}}\}=%
e_{\mathbf L}(F(x)_{\alpha=}).
\end{array}
$$ 

Therefore $e_{\mathbf L}(x_{\alpha S^{\bullet}}\vee F(x)_{\alpha=})=e_{\mathbf L}(F(x)_{\alpha=})$, which assures 
by  Lemma \ref{lemprenucleus2} the existence of $\psi_{\mathbf L}$ from the theorem. 

Let  $g\colon \mathbf L_1\to \mathbf L_2$  be a morphism of sup-semilattices. Let us 
show that the following diagram commutes:
\begin{center}
\begin{tikzpicture}[scale=0.68049]
\node (alfa') at (7,5) {${\mathbf J}[\mathbf H, \mathbf L_1]\otimes \mathbf H$};
  \node (beta') at (7,0) {$\mathbf L_1$};
   \node (gama') at (14,0) {$\mathbf L_2$};
   \node (delta') at (14,5) {${\mathbf J}[\mathbf H,\mathbf L_2]\otimes \mathbf H$};
   \node (epsilon') at (3,8) {$\mathbf G^{T_{[\mathbf H,\mathbf L_1]}}$};
   \node (mu') at (18,8) {$\mathbf G^{T_{[\mathbf H,\mathbf L_2]}}$};
   
   \node (1') at (10.5,5.5) {${\mathbf J}[\mathbf H,g]\otimes \mathbf H$};
   \node (2') at (10.5,0.5) {$g$};
   \node (3') at (7.6,2.5) {$\psi_{{\mathbf L}_1}$};
   \node (4') at (13.3,2.5) {$\psi_{\mathbf L_2}$};
    \node (5') at (7.098595,6.6510595) {$\mbox{\rm n}(j[\mathbf J[\mathbf H, \mathbf L_1],{\mathbf H}])$};
       \node (6') at (5.09876,2.5) {$f_{{\mathbf L}_1}$};
    \node (7') at (13.8595,6.6510595) {$\mbox{\rm n}(j[\mathbf J[\mathbf H, \mathbf L_2],{\mathbf H}])$};
           \node (8') at (15.876509876,2.5) {$f_{{\mathbf L}_2}$};
        \node (9') at (10.5,8.5) {${\mathbf J}[\mathbf H,g]^{\rightarrow}$};      
   
  \draw [->](alfa') -- (delta');
  \draw [->](beta') -- (gama');
  \draw [->](delta') -- (gama');
  \draw [->](alfa') -- (beta');
   \draw [->](epsilon') -- (alfa');
      \draw [->](epsilon') -- (beta');
       \draw [->](mu') -- (delta');
   \draw [->](mu') -- (gama');
    \draw [->](epsilon') -- (mu');
 \end{tikzpicture}
 \end{center}

Let $x\in G^{T_{[\mathbf H,\mathbf L_1]}}$. We compute: 
$$
\begin{array}{@{}r@{\,}c@{\,}l}
(\psi_{\mathbf L_2}&\circ& ({\mathbf J}[\mathbf H,g]\otimes \mathbf H)\circ %
\mbox{\rm n}(j[\mathbf J[\mathbf H, \mathbf L_1],{\mathbf H}]))(x)\\[0.2cm]
&=&  %
(\psi_{\mathbf L_2}\circ \mbox{\rm n}(j[\mathbf J[\mathbf H, \mathbf L_2],{\mathbf H}])\circ %
{\mathbf J}[\mathbf H,g]^{\rightarrow})(x) =%
(f_{{\mathbf L}_2}\circ {\mathbf J}[\mathbf H,g]^{\rightarrow})(x)\\[0.2cm]
&=&%
f_{{\mathbf L}_2}((\bigvee\{x(\alpha) \mid g\circ \alpha=\beta, %
\alpha \in T_{[\mathbf H,\mathbf L_1]}\})_{\beta \in T_{[\mathbf H,\mathbf L_2]}})\\[0.2cm]
&=&\bigvee _{\beta \in T_{[\mathbf H,\mathbf L_2]}} %
 \beta((\bigvee\{x(\alpha) \mid g\circ \alpha=\beta, %
\alpha \in T_{[\mathbf H,\mathbf L_1]}\}))\\[0.2cm]
&=&
 \bigvee\{g(\alpha(x(\alpha))) \mid  %
\alpha \in T_{[\mathbf H,\mathbf L_1]}\}=
g(\bigvee\{\alpha(x(\alpha)) \mid  \alpha \in T_{[\mathbf H,\mathbf L_1]}\})\\[0.2cm]
&=&(g\circ f_{{\mathbf L}_1})(x)=(g\circ %
\psi_{\mathbf L_1}\circ \mbox{\rm n}(j[\mathbf J[\mathbf H, \mathbf L_1],{\mathbf H}])(x).
\end{array}
$$ 
Since $\mbox{\rm n}(j[\mathbf J[\mathbf H, \mathbf L_1],{\mathbf H}])$ is surjective 
we have 
$\psi_{\mathbf L_2}\circ ({\mathbf J}[\mathbf H,g]\otimes \mathbf H)=%
g\circ \psi_{\mathbf L_1}$.

(c): Let  $\mathbf J=(S, T)$ be a frame and  $\mathbf L$   a  sup-semilattice.
We will prove the commutativity of following diagrams:

\noindent{}%
\begin{tabular}{@{}c c@{}c}
\begin{tikzpicture}[scale=0.3]

\node (alfa') at (12,0) {$\mathbf  J\otimes \mathbf  H$};
  
   \node (gama') at (0,5) {$\mathbf  J\otimes\mathbf  H$};
   \node (delta') at (12,5) {${\mathbf J}[\mathbf H, \mathbf  J\otimes\mathbf  H]\otimes \mathbf H$};
   
   \node (1') at (5.5,5.7) {$\varphi_{\mathbf  J}\otimes {\mathbf  H}$};
   \node (2') at (13.7,2.7) {$\psi_{\mathbf  J\otimes\mathbf  H}$};
   \node (3') at (6,1.5) {$\mbox{\rm id}_{\mathbf  J\otimes \mathbf  H}$};

  \draw [->](gama') -- (alfa');
  \draw [->](delta') -- (alfa');
  \draw [->](gama') -- (delta');
  
 \end{tikzpicture}
 &&
 \begin{tikzpicture}[scale=0.3]

\node (alfa') at (12,0) {${\mathbf J}[\mathbf H, {\mathbf L}]$};
  
   \node (gama') at (0,5) {${\mathbf J}[\mathbf H, {\mathbf L}]$};
   \node (delta') at (12,5) {${\mathbf J}[\mathbf H, {\mathbf J}[\mathbf H, \mathbf  L]\otimes \mathbf H]$};
   
   \node (1') at (5.6,5.7) {$\varphi_{{\mathbf J}[\mathbf H, {\mathbf L}]}$};
   \node (2') at (14.1099,2.5) {${{\mathbf J}[\mathbf H, \psi_{\mathbf L}]}{}$};
   \node (3') at (6,1.5) {$\mbox{\rm id}_{{\mathbf J}[\mathbf H, {\mathbf L}]}$};

  \draw [->](gama') -- (alfa');
  \draw [->](delta') -- (alfa');
  \draw [->](gama') -- (delta');
  
 \end{tikzpicture}
\end{tabular}

Let $\overline{x}\in \mathbf G^T$.  According to  Theorem \ref{Jtensor}
we know that the following diagram commutes:

\begin{center}
\begin{tikzpicture}[scale=0.62347]
\node (alfa') at (6,5) {$\mathbf G^{T}$};
  \node (beta') at (6,0) {$\mathbf G^{T_{[\mathbf H, \mathbf J\otimes \mathbf H]}}$};
   \node (gama') at (15,0) {${\mathbf J}[\mathbf H, \mathbf J\otimes \mathbf H]\otimes \mathbf H$};
   \node (delta') at (15,5) {$\mathbf J\otimes \mathbf H$};
    \node (deltax') at (22,0) {$\mathbf J\otimes \mathbf H$};
      \node (deltaxz') at (22,-3.5) {$\mathbf J\otimes \mathbf H$};
   
   \node (1') at (10.5,5.5) {$\mbox{\rm n}(j[\mathbf J,\mathbf H])$};
   \node (2') at (10,0.5) {$\mbox{\rm n}(j[{\mathbf J}[\mathbf H, \mathbf J\otimes \mathbf H], %
   \mathbf H])$};
   \node (3') at (6.85,2.5) {$\varphi _{\mathbf J}^{\rightarrow}$};
   \node (4') at (13.9,2.5) {$\varphi _{\mathbf J}\otimes\mathbf  H$};
      \node (5') at (18.5,0.5) {$\psi _{\mathbf J\otimes\mathbf  H}$};
       \node (6') at (20.9,-1.5) {$\mbox{\rm id}_{\mathbf  J\otimes \mathbf  H}$};
        \node (7') at (13.59,-2.85) {$f_{\mathbf J\otimes {\mathbf H}}$};
         \node (8') at (19.9,2.5) {$\mbox{\rm id}_{\mathbf  J\otimes \mathbf  H}$};
   
  \draw [->](alfa') -- (delta');
  \draw [->](beta') -- (gama');
  \draw [->](delta') -- (gama');
  \draw [->](alfa') -- (beta');
    \draw [->](gama') -- (deltax');
     \draw [->](deltax') -- (deltaxz');
        \draw [->](beta') -- (deltaxz');
        \draw[densely dashed,->] (delta') --  (deltax');
 \end{tikzpicture}
 \end{center}
\vskip-0.3cm

We compute:

$$
\begin{array}{@{}r@{\,}c@{\,}l}
(&\psi_{\mathbf J\otimes\mathbf  H}&\circ (\varphi _{\mathbf J}\otimes\mathbf  H)\circ %
\mbox{\rm n}(j[\mathbf J,\mathbf H]))(x)=
(\psi _{\mathbf J\otimes\mathbf  H}\circ%
\mbox{\rm n}(j[{\mathbf J}[\mathbf H, \mathbf J\otimes \mathbf H], %
   \mathbf H])\circ %
\varphi _{\mathbf J}^{\rightarrow})(x)\\[0.2cm]
&\multicolumn{2}{@{}l}{=(f_{\mathbf J\otimes {\mathbf H}}\circ %
\varphi _{\mathbf J}^{\rightarrow})(x)=%
f_{\mathbf J\otimes {\mathbf H}}%
\left(\bigvee\{x(i)\mid \varphi_{\mathbf J}(i)=\alpha\})_{\alpha\in T_{[\mathbf H, %
\mathbf J\otimes \mathbf H]}}\right)}\\[0.2cm]
&\multicolumn{2}{@{}l}{=\bigvee_{\alpha\in T_{[\mathbf H, %
\mathbf J\otimes \mathbf H]}}%
\alpha(\bigvee\{x(i)\mid \varphi_{\mathbf J}(i)=\alpha\})%
=\bigvee_{\alpha\in T_{[\mathbf H, %
\mathbf J\otimes \mathbf H]}}%
\bigvee_{\varphi_{\mathbf J}(i)=\alpha, i\in T}\alpha(x(i))}\\[0.2cm]
&\multicolumn{2}{@{}l}{=\bigvee\{(\varphi_{\mathbf J}(i))(x(i)) \mid 
{i\in T}\}=%
\bigvee\{\mbox{\rm n}(j[\mathbf J,{\mathbf H}])(x(i)_{i=}) \mid 
{ i\in T}\}}\\[0.2cm]
&\multicolumn{2}{@{}l}{=\mbox{\rm n}(j[\mathbf J,{\mathbf H}])(\bigvee\{ x(i)_{i=} \mid  i\in T\})=%
\mbox{\rm n}(j[\mathbf J,{\mathbf H}])(x).}
\end{array}
$$

Hence the first diagram commutes. Now, let  $\alpha \in {\mathbf J}[\mathbf H,\mathbf L]$ and  $x\in G$. 
We obtain: 
$$
\begin{array}{@{}r@{\,}l@{\,}l}
(({{\mathbf J}[\mathbf H, \psi_{\mathbf L}]}{} &\circ \varphi_{{\mathbf J}[\mathbf H, {\mathbf L}]})(\alpha))(x)&=%
({{\mathbf J}[\mathbf H, \psi_{\mathbf L}]}{}(\varphi_{{\mathbf J}[\mathbf H, {\mathbf L}]}(\alpha))(x) %
\\[0.2cm]
&\multicolumn{2}{@{}l}{%
=\psi_{\mathbf L}((\varphi_{{\mathbf J}[\mathbf H, {\mathbf L}]}(\alpha))(x))=\psi_{\mathbf L}(\mbox{\rm n}(j[\mathbf J[\mathbf H, {\mathbf L}],{\mathbf H}])(x_{\alpha=}))=%
f_{\mathbf L}(x_{\alpha=})}\\[0.2cm]
&\multicolumn{2}{@{}l}{=\bigvee\{\beta(x_{\alpha=}(\beta)) \mid \beta\in T_{[\mathbf H, {\mathbf L}]}\}=\alpha (x)}
\end{array}$$
which yields the commutativity of the second diagram.
\end{proof}

\begin{theorem}\label{natsupsem}
Let $\mathbf L$ be a sup-semilattice. Then the following holds:
\begin{enumerate}[\rm(a)]
\item For an arbitrary frame $\mathbf J=(T,S)$, there exists a unique homomorphism of frames 
$\nu_{\mathbf J}\colon \mathbf J\to {\mathbf J}[{\mathbf L}^{\mathbf J}, \mathbf L]$ defined for 
arbitrary $x\in L^T$ and $i\in T$ in such a way that 
$$(\nu_{\mathbf J}(i))(x)=x({i}).$$ 
Moreover,  $\nu=(\nu_{\mathbf J}\colon %
{\mathbf J}\to {\mathbf J}[{\mathbf L}^{\mathbf J}, \mathbf L])_{\mathbf J\in {\mathbb J}}$ 
is a natural transformation  between the identity functor on  $\mathbb J$ and the endofunctor 
${\mathbf J}[{\mathbf L}^{-}, \mathbf L]$.

\item  For an arbitrary $F$-sup-semilattice $\mathbf H=(\mathbf G, F)$ there exists a lax morphism 
$\mu_{\mathbf H}\colon {\mathbf H}\to %
\mathbf L^{\mathbf J[\mathbf H, \mathbf L]}$ 
of $F$-sup-semilattices defined for arbitrary $x\in G$ and $\alpha\in T_{\mathbf J[\mathbf H, \mathbf L]}$ 
by  
$$(\mu_{\mathbf H}(x))(\alpha)=\alpha(x).$$ 
Moreover, 
$\mu=(\mu_{\mathbf H}\colon {\mathbf H}\to %
\mathbf L^{\mathbf J[\mathbf H, \mathbf L]})_{{\mathbf H}\in \smFSUPL}$ 
is a natural transformation  between the identity functor on  ${\FSUPL}$ and the endofunctor 
$\mathbf L^{\mathbf J[-, \mathbf L]}$. 

\item There exists an adjoint situation 
$(\nu, \mu)\colon {\mathbf  J}[-, \mathbf  L])\dashv {\mathbf  L}^{-}\colon %
\mathbb J \to {\FSUPL}^{op}$.

\end{enumerate}
\end{theorem}
\begin{proof}
{(a):} Let $i, k\in T$, $x\in L^{T}$ and $i\mathrel{S} k$. We compute:
$$
(\nu_{\mathbf J}(k))(x)=x({k})\leq \bigvee\{ x(l)\mid i\mathrel{S} l, l\in T\}=%
(F^{\mathbf J}(x))(i)=(\nu_{\mathbf J}(i))(F^{\mathbf J}(x)).
$$
Hence $\nu_{\mathbf J}(i) \mathrel{S_{[{\mathbf L}^{\mathbf J}, \mathbf L]}} \nu_{\mathbf J}(k)$ 
and $\nu_{\mathbf J}$ is a frame homomorphism. 
Assume now that $t\colon \mathbf J_1 \to \mathbf J_2$ is a frame homomorphism between 
frames $\mathbf J_1=(T_1,S_1)$ and $\mathbf J_2=(T_2,S_2)$. We have to show 
that the following diagram commutes:
\begin{center}
\begin{tikzpicture}

\node (alfa') at (7,5) {$\mathbf J_1$};
  \node (beta') at (7,2) {${\mathbf J}[{\mathbf L}^{\mathbf J_1}, \mathbf L]$};
   \node (gama') at (13,2) {${\mathbf J}[{\mathbf L}^{\mathbf J_2}, \mathbf L]$};
   \node (delta') at (13,5) {$\mathbf J_2$};
   
   \node (1') at (9.5,4.5) {$t$};
   \node (2') at (10.05,2.5) {${\mathbf J}[{\mathbf L}^{t},  \mathbf L]$};
   \node (3') at (7.5483,3.55) {$\nu_{\mathbf J_1}$};
   \node (4') at (12.42,3.55) {$\nu_{\mathbf J_2}$};
   
  \draw [->](alfa') -- (delta');
  \draw [->](beta') -- (gama');
  \draw [->](delta') -- (gama');
  \draw [->](alfa') -- (beta');
 \end{tikzpicture}
\end{center}

Assume that $i\in T_1$ and $x\in  L^{T_2}$. We compute:
$$
\begin{array}{@{}r@{\,}c@{\,}l}
\left(({\mathbf J}[{\mathbf L}^{t},  \mathbf L]\circ \nu_{\mathbf J_1})(i)\right)(x)&=&%
({\mathbf J}[{\mathbf L}^{t},  \mathbf L]( \nu_{\mathbf J_1}(i))(x)=%
( \nu_{\mathbf J_1}(i)\circ {\mathbf L}^{t})(x)=%
\nu_{\mathbf J_1}(i)({\mathbf L}^{t}(x))\\[0.2cm]
&=&%
\nu_{\mathbf J_1}(i)(x\circ t)=(x\circ t)(i)=x(t(i))=%
 (\nu_{\mathbf J_2}(t(i))(x)\\[0.2cm]
&=&%
\left((\nu_{\mathbf J_2}\circ t)(i)\right)(x).\\
\end{array}
$$
Hence ${\mathbf J}[{\mathbf L}^{t},  \mathbf L]\circ \nu_{\mathbf J_1}=%
\nu_{\mathbf J_2}\circ t$.

{(b):}  Evidently, $\mu_{\mathbf H}$ preserves arbitrary joins. We have to verify that 
$F^{{\mathbf J[\mathbf H, \mathbf L]}}\circ \mu_{\mathbf H}\leq \mu_{\mathbf H}\circ F$. 
Let $x\in G$ and $\alpha\in {{\mathbf J[\mathbf H, \mathbf L]}}$. We compute: 
$$
\begin{array}{@{}r@{\,}c@{\,}l}
\left((F^{{\mathbf J[\mathbf H, \mathbf L]}}\circ \mu_{\mathbf H})(x)\right)(\alpha)&=&
(F^{{\mathbf J[\mathbf H, \mathbf L]}}(\mu_{\mathbf H}(x)))(\alpha)\\[0.2cm]
&=&%
\bigvee\{(\mu_{\mathbf H}(x))(\beta) \mid \beta\in {{\mathbf J[\mathbf H, \mathbf L]}}, 
\alpha \mathrel{S_{[\mathbf H, \mathbf L]}} \beta\}\\[0.2cm]
&=&%
\bigvee\{\beta(x) \mid \beta\in {{\mathbf J[\mathbf H, \mathbf L]}}, 
\alpha \mathrel{S_{[\mathbf H, \mathbf L]}} \beta\}\\[0.2cm]
&\leq&%
\alpha(F(x))=\mu_{\mathbf H}(F(x))(\alpha)=%
\left((\mu_{\mathbf H}\circ F) (x)\right)(\alpha).
\end{array}
$$

Therefore is $\mu_{\mathbf H}$ a lax morphism. 

Now, let us assume that 
$\mathbf H_1=(\mathbf G_1, F_1)$ and  %
$\mathbf H_2=(\mathbf G_2, F_2)$ are $F$-sup-semilattices, and that 
$f\colon \mathbf H_1\to \mathbf  H_2$ is a  lax morphism of $F$-sup-semilattices. 
We have to verify that the following diagram commutes:
\begin{center}
\begin{tikzpicture}

\node (alfa') at (7,5) {$\mathbf H_1$};
  \node (beta') at (7,2) {$\mathbf L^{\mathbf J[\mathbf H_1, \mathbf L]}$};
   \node (gama') at (12,2) {$\mathbf L^{\mathbf J[\mathbf H_2, \mathbf L]}$};
   \node (delta') at (12,5) {$\mathbf H_2$};
   
   \node (1') at (9.5,4.5) {$f$};
   \node (2') at (9.5,2.5) {$\mathbf L^{\mathbf J[f, \mathbf L]}$};
   \node (3') at (7.5,3.55) {$\mu_{\mathbf H_1}$};
   \node (4') at (11.5,3.55) {$\mu_{\mathbf H_2}$};
   
  \draw [->](alfa') -- (delta');
  \draw [->](beta') -- (gama');
  \draw [->](delta') -- (gama');
  \draw [->](alfa') -- (beta');
 \end{tikzpicture}
 \end{center}
 Assume that $x\in G_1$ and $\alpha\in T_{[\mathbf H_2, \mathbf L]}$. We compute: 
 
 $$
\begin{array}{@{}r@{}l@{}l}
\left((\mathbf L^{\mathbf J[f, \mathbf L]} \right.&\left.\circ \mu_{\mathbf H_1})(x)\right)(\alpha)&=%
\left(\mathbf L^{\mathbf J[f, \mathbf L]} (\mu_{\mathbf H_1}(x))\right)(\alpha)\\[0.2cm]
&\multicolumn{2}{@{}l}{=%
\left(\mu_{\mathbf H_1}(x) \right)({\mathbf J[f, \mathbf L](\alpha)})=%
({\mathbf J[f, \mathbf L](\alpha)})(x)%
=(\alpha\circ f)(x)}\\[0.2cm]
&\multicolumn{2}{@{}l}{=%
\alpha(f(x))=\left((\mu_{\mathbf H_2}\circ f)(x)\right)(\alpha).}%
\end{array}
$$
hence $\mathbf L^{\mathbf J[f, \mathbf L]} \circ \mu_{\mathbf H_1}=\mu_{\mathbf H_2}\circ f$. 
\end{proof}

 (c): Let  $\mathbf J=(S, T)$ be a frame and  $\mathbf H=(\mathbf G, F)$ an $F$-sup-semilattice.
We will prove the commutativity of following diagrams:

\begin{center}
\begin{tabular}{@{}c c c}
\begin{tikzpicture}[scale=0.3]

\node (alfa') at (12,0) {$\mathbf  L^{\mathbf  J}$};
  
   \node (gama') at (0,5) {$\mathbf  L^{\mathbf  J}$};
   \node (delta') at (12,5) {${\mathbf L}^{{\mathbf J}%
   [\mathbf L^{\mathbf J}, \mathbf  L]}$};
   
   \node (1') at (5.5,5.7) {$\mu_{{\mathbf  L}^{\mathbf  J}}$};
   \node (2') at (13.7,2.7) {${\mathbf  L}^{\nu_{\mathbf  J}}$};
   \node (3') at (6,1.5) {$\mbox{\rm id}_{\mathbf  L^{\mathbf  J}}$};

  \draw [->](gama') -- (alfa');
  \draw [->](delta') -- (alfa');
  \draw [->](gama') -- (delta');
  
 \end{tikzpicture}
 &\phantom{xx}&
 \begin{tikzpicture}[scale=0.3]

\node (alfa') at (12,0) {${\mathbf J}[\mathbf H, {\mathbf L}]$};
  
   \node (gama') at (0,5) {${\mathbf J}[\mathbf H, {\mathbf L}]$};
   \node (delta') at (12,5) {${\mathbf J}[\mathbf L^{{\mathbf J}[\mathbf H, \mathbf  L]}, \mathbf L]$};
   
   \node (1') at (5.6,5.7) {$\nu_{{\mathbf J}[\mathbf H, {\mathbf L}]}$};
   \node (2') at (14.3099,2.5) {${{\mathbf J}[\mu_{\mathbf H}, {\mathbf L}]}{}$};
   \node (3') at (5.2469,1.5) {$\mbox{\rm id}_{{\mathbf J}[\mathbf H, {\mathbf L}]}$};

  \draw [->](gama') -- (alfa');
  \draw [->](delta') -- (alfa');
  \draw [->](gama') -- (delta');
  
 \end{tikzpicture}
\end{tabular}
\end{center}

Let $x\in L^{T}$ and $i\in T$. We compute: 
 $$
\begin{array}{@{}r@{\,}l@{\,}l}
\left(({\mathbf  L}^{\nu_{\mathbf  J}} \circ \mu_{{\mathbf  L}^{\mathbf  J}})(x)\right)(i)&=&%
\left({\mathbf  L}^{\nu_{\mathbf  J}} (\mu_{{\mathbf  L}^{\mathbf  J}}(x))\right)(i)=%
(\mu_{{\mathbf  L}^{\mathbf  J}}(x))\left({\nu_{\mathbf  J}} (i)\right)=%
\left({\nu_{\mathbf  J}} (i)\right)(x)\\[0.2cm]
&=&%
x(i)=(\mbox{\rm id}_{\mathbf  L^{\mathbf  J}}(x))(i).
\end{array}
$$
Hence ${\mathbf  L}^{\nu_{\mathbf  J}} \circ \mu_{{\mathbf  L}^{\mathbf  J}}=%
\mbox{\rm id}_{\mathbf  L^{\mathbf  J}}$ in $\FSUPL$. Assume that $x\in G$ and $\alpha\in T_{[\mathbf H, \mathbf L]}$. We compute: 
 $$
\begin{array}{@{}r@{\,}c@{\,}l}
\left(({{\mathbf J}[\mu_{\mathbf H}, {\mathbf L}]}{}\right. &\left.\circ\, \, %
\nu_{{\mathbf J}[\mathbf H, {\mathbf L}]})(\alpha)\right)(x)&=%
\left({{\mathbf J}[\mu_{\mathbf H}, {\mathbf L}]}{} (\nu_{{\mathbf J}[\mathbf H, {\mathbf L}]}(\alpha))\right)(x)
\\[0.2cm]
&\multicolumn{2}{@{}l}{=%
(\nu_{{\mathbf J}[\mathbf H, {\mathbf L}]}(\alpha))\left(\mu_{\mathbf H}(x)\right)%
=\left(\mu_{\mathbf H}(x)\right)(\alpha)=%
\alpha(x)=(\mbox{\rm id}_{{\mathbf J}[{\mathbf H}, {\mathbf L}]}(\alpha))(x).}
\end{array}
$$
Therefore ${{\mathbf J}[\mu_{\mathbf H}, {\mathbf L}]}{} \circ %
\nu_{{\mathbf J}[\mathbf H, {\mathbf L}]}=\mbox{\rm id}_{{\mathbf J}[{\mathbf H}, {\mathbf L}]}$. 

\begin{remark}
Our induced adjoint situations $(\eta, \varepsilon), (\varphi, \psi)$,  and $(\nu, \mu)$ 
evidently restrict to an adjoint situation between the category of finite 
sup-semilattices and the category of  finite $F$-sup-semilattices, between the category of finite 
sup-semilattices and  the category of finite frames, and the category of finite frames and the dual of 
the category  of  finite $F$-sup-semilattices, respectively. 
\end{remark}


\section{Three examples}\label{approaches}


In this section we will analyze three examples to illustrate adjoint situations 
$(\eta, \varepsilon), (\varphi, \psi)$  and $(\nu, \mu)$.

\begin{example}\label{firstexample}
Let ${\mathbf H}=({\mathbf G},F)$ be an 
$F$-sup-semilattice, ${\mathbf G}=(\{0,a,b,c,1\},\bigvee)$ and $0<a,b,c< 1$. 
Operator $F$ is given by the prescription  $F(0)=0$, $F(a)=a$, $F(b)=c$, $F(c)=b$, $F(1)=1$. 
Let ${\mathbf L}=(\{0, 1\}, \bigvee\}$ be a sup-semilattice where $0< 1$. Both sup-semilattices 
are shown in the following figures:

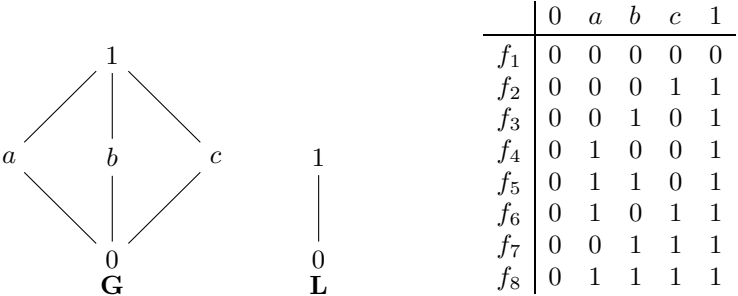
\begin{center}
\begin{tabular}{l c r}
\begin{tikzpicture}[scale=0.68]
  \node (1_2) at (7,6) {$1$};
  \node (b2) at (9,4) {$c$};
  \node (a2) at (5,4) {$a$};
  \node (c2) at (7,4) {$b$};
  \node (0_2) at (7,2) {$0$};
  \node (10_2) at (7,1.5) {${\mathbf G}$};
  \draw (1_2) -- (b2) -- (0_2) -- (a2) -- (1_2) -- (c2) -- (0_2);
  
  \node (1_1) at (11,4) {$1$};
  \node (0_1) at (11,2) {$0$};
   \node (100_2) at (11,1.5) {${\mathbf L}$};
  \draw (1_1) -- (0_1);
  \end{tikzpicture}&\phantom{xxxxxx} &
  \begin{tikzpicture}  
  \node (11_2) at (14,6) {\begin{tabular}{l | lllll}
     & $0$ & $a$ & $b$ & $c$ & $1$  \\[0.01cm] \hline
     $$ &  &  &  &  &   \\[-0.351cm]
$f_1$ & 0 & 0 & 0 & 0 & 0  \\
$f_2$ & 0 & 0 & 0 & 1 & 1  \\
$f_3$ & 0 & 0 & 1 & 0 & 1  \\
$f_4$ & 0 & 1 & 0 & 0 & 1  \\
$f_5$ & 0 & 1 & 1 & 0 & 1  \\
$f_6$ & 0 & 1 & 0 & 1 & 1  \\
$f_7$ & 0 & 0 & 1 & 1 & 1  \\
$f_8$ & 0 & 1 & 1 & 1 & 1 
\end{tabular}};
 \end{tikzpicture}
  \end{tabular}
\end{center}
Let us define a frame ${{\mathbf J}{[{\mathbf H},{\mathbf L}]}}%
=({\mathbb S}({\mathbf G},{\mathbf L}), %
S_{[{\mathbf H},{\mathbf L}]})$ where $S_{[{\mathbf H},{\mathbf L}]}$ 
is a relation from Definition \ref{framefromfss}. Let us denote $S_{[{\mathbf H},{\mathbf L}]}$ 
as $\rho$. Clearly, ${\mathbb S}({\mathbf G},{\mathbf L})$ has 8 elements, which we will denote $f_i$, where $i\in \{1,2,3,4,5,6,7,8\}$ and their description is given by the preceding table. Moreover, 
$f_1\leq f_2\leq f_6, f_7\leq f_8$ and $f_1\leq f_3\leq f_5, f_7\leq f_8$ and  
$f_1\leq f_4\leq f_5, f_6\leq f_8$. 

Let us now describe the relation $\rho$ on ${\mathbb S}({\mathbf G},{\mathbf L})$. By definition,  

$$f_i\mathrel{\rho} f_j \iff \forall x\in G\ \  f_j(x)\leq f_i(F(x)).$$

Clearly $f_8\rho f_i$ and $f_i\rho f_1$ for any $i\in \{1,2,3,4,5,6,7,8\}$. Furthermore, it is easy to see that  $f_1\circ F=f_1$, $f_2\circ F=f_3$, $f_3\circ F=f_2$, $f_4\circ F=f_4$, $f_5\circ F=f_6$, $f_6\circ F=f_5$, $f_7\circ F=f_7$ and $f_8\circ F=f_8$. This means $f_2\rho f_3$, $f_3\rho f_2$, $f_4\rho f_4$, $f_5\rho f_6$, $f_6\rho f_5$, $f_7\rho f_7$. Hence also $f_5\rho f_2$, $f_5\rho f_4$, $f_6\rho f_3$, $f_6\rho f_4$, 
$f_7\rho f_2$, $f_7\rho f_3$. We obtain that 
$$
\begin{array}{r@{\,\,}c@{}l}
\rho&=\{&(f_8, f_1), (f_8, f_2), (f_8, f_3), (f_8, f_4), (f_8, f_5), (f_8, f_6),  %
 (f_7, f_1), (f_6, f_1),\\[0.1cm]
&\multicolumn{2}{@{}l}{(f_5, f_1), (f_4, f_1), (f_3, f_1), (f_2, f_1), (f_1, f_1), (f_2, f_3), (f_3, f_2), 
(f_5, f_6), (f_6, f_5),}\\[0.1cm]
&\multicolumn{2}{@{}l}{(f_5, f_2), (f_5, f_4), (f_6, f_3), (f_6, f_4), (f_7, f_2), (f_7, f_3), (f_8, f_7), %
(f_8, f_8), (f_7, f_7)\}.}%
\end{array}
$$

By Theorem \ref{natsupsem}, there exists a lax morphism 
$\mu _{\mathbf H}\colon{}{\mathbf H} \longrightarrow %
{\mathbf L}^{{\mathbf J}{[{\mathbf H},{\mathbf L}]}}$ of $F$-sup-semilattices 
defined for arbitrary $x\in G$ and $f_i\in {\mathbb S}({\mathbf G},{\mathbf L})$ by

$$(\mu _{\mathbf H}(x))(f_i)=f_i(x).$$

Let us now compute $\mu _H$ on elements of $G$. It holds that:

$$(\mu _{\mathbf H}(0))(f_i) = 0 \text{ and } (\mu _{\mathbf H}(1))(f_i) = 1 
\text{ for any } i\in \{1,2,3,4,5,6,7,8\}, $$

\[ 
\begin{array}{c@{}c@{}c}
(\mu _{\mathbf H}(a))(f_i) =
  \begin{cases}
    1       &\text{for }  i\in \{4,5,6,8\}\\
    0 &\text{otherwise}, 
  \end{cases}& &\ \,
  (\mu _{\mathbf H}(b))(f_i) =
  \begin{cases}
    1       &\text{for }  i\in \{3,5,7,8\}\\
    0 &\text{otherwise}, 
  \end{cases}

  \end{array}
\]
\[ (\mu _{\mathbf H}(c))(f_i) =
  \begin{cases}
    1       & \text{for }  i\in \{2,6,7,8\}\\
    0 & \text{otherwise}. 
  \end{cases}
\]

Evidently, $\mu_{\mathbf H}$ is injective. Recall that the ordering of elements of 
${\mathbf L}^{{\mathbf J}{[{\mathbf H},{\mathbf L}]}}$ is given by:

$$\alpha\leq\beta\iff %
\forall i\in \{1,2,3,4,5,6,7,8\}\ \alpha(f_i)\leq \beta(f_i).$$

Since $\mu _{\mathbf H}$ preserves arbitrary joins and it is injective we have immediately that 
$\mu _{\mathbf H}$  is an order embedding.

Let us now show that $\mu_{\mathbf H}$ is not only a lax morphism 
but a homomorphism of $F$-sup-semilattices as well. In other words, 
we have to verify that 

$$F^{{\mathbf J}{[{\mathbf H},{\mathbf L}]}}%
\circ \mu_{\mathbf H}= \mu_{\mathbf H}\circ F.$$ 

Since $\mu_{\mathbf H}$ is  a lax morphism we have that 

$$F^{{\mathbf J}{[{\mathbf H},{\mathbf L}]}}\circ %
\mu_{\mathbf H}\leq \mu_{\mathbf H}\circ F,$$

holds in general so let us prove the reverse inequality. Since $F$ is a bijective 
morphism of sup-semilattices, 
 for any $f_i$ there exists 
$f_{j(i)}$ such that $f_{j(i)}\circ F=f_i$ and hence $f_i\rho f_{j(i)}$. 
We then have 
$$
\begin{array}{r c l}
((\mu_{\mathbf H}\circ F)(x))(f_i)&=&\mu_{\mathbf H}(F(x))(f_i)=f_i(F(x))=f_{j(i)}(x)\\
&\leq& 
\bigvee\{f_j(x) \mid f_j\in {\mathbb S}({\mathbf G},{\mathbf L}),f_i \rho f_j\}\\
&=&\bigvee\{(\mu_{\mathbf H}(x))(f_j) \mid f_j\in {\mathbb S}(G,L), f_i \rho f_j\}\\
&=&(F^{{\mathbf J}{[{\mathbf H},{\mathbf L}]}}(\mu_{\mathbf H}(x)))(f_i)\\
&=&\left((F^{{\mathbf J}{[{\mathbf H},{\mathbf L}]}}\circ \mu_{\mathbf H})(x)\right)(f_i)
\end{array}$$

since $f_i(x)\in \{f_j(x) \mid f_j\in {\mathbb S}({\mathbf G},{\mathbf L}),f_i \rho f_j\}$.\\

We conclude that $\mu_{\mathbf H}$ is indeed a F-sup-semilattice homomorphism and therefore it is an element of $\mathcal{E}_\leq$ since $\mathcal{E}\subseteq\mathcal{E}_\leq$.

\end{example}

\begin{example}\label{secondexample}
Let ${\mathbf J}=(T,S)$ be a frame where 
$T=\{f_2,f_3,f_4\}$, $f_2,f_3,f_4$ are mappings from Example  \ref{firstexample} and $S$ is a restriction 
of the relation $\rho$ on $T$  from Example  \ref{firstexample}. Hence 
 $S$ is a relation on $\{f_2,f_3,f_4\}$ given by $f_2 \mathrel{S} f_3$, $f_3\mathrel{S} f_2$ and 
  $f_4\mathrel{S} f_4$. Let ${\mathbf L}=(\{0, 1\}, \bigvee\}$ be a sup-semilattice where $0< 1$.

By Theorem \ref{adjSFSL} there exists a unique homomorphism 
$\nu_{\mathbf J}\colon \mathbf J\to {\mathbf J}[{\mathbf L}^{\mathbf J}, \mathbf L]$ 
of frames  defined for 
arbitrary $x\in L^T$ and $f_i\in T$ in such a way that 
$$(\nu_{\mathbf J}(f_i))(x)=x({f_i}).$$ 

First of all, ${L^T}$ is obviously an 8-element set and we denote its elements as $\alpha _k$, $k\in \{1,2,3,4,5,6,7,8\}$ (see Table \ref{tab:label}). Let us now compute the induced tense operator ${\mathbf F^{\mathbf J}}$ on ${\mathbf L^{\mathbf J}}$. By its definition, ${\mathbf F^{\mathbf J}}$  is given by

$$({{\mathbf F^{\mathbf J}}}(\alpha _k))(f_i)=\bigvee \{\alpha _k(f_l)\mid f_i\mathrel{S} f_l\},$$

for $k\in \{1,2,3,4,5,6,7,8\}$ and $i\in \{2,3,4\}$. Since $f_2\mathrel{S} f_3$, 
$f_3\mathrel{S} f_2$ and $f_4\mathrel{S} f_4$, it is easy to see that 
$({\mathbf F^{\mathbf J}}(\alpha _k))(f_2)=\alpha _k(f_3)$, 
$({\mathbf F^{\mathbf J}}(\alpha _k))(f_3)=\alpha _k(f_2)$ and 
$({\mathbf F^{\mathbf J}}(\alpha _k))(f_4)=\alpha _k(f_4)$. Using this fact, we can evaluate for ${\mathbf F^{\mathbf J}}(\alpha _k)$ 
for all $k\in \{1,2,3,4,5,6,7,8\}$ and $i\in \{2,3,4\}$, which is shown in Table \ref{tab:label}.

\begin{table}[ht]
\centering
\begin{tabular}{c c c}
\begin{tabular}{c}
{\begin{tabular}{l | lll}
     & $f_2$ & $f_3$ & $f_4$   \\[0.01cm] \hline
     &  &  &   \\[-0.351cm]
$\alpha _1$ & 0 & 0 & 0  \\
$\alpha _2$ & 1 & 0 & 0  \\
$\alpha _3$ & 0 & 1 & 0  \\
$\alpha _4$ & 0 & 0 & 1  \\
$\alpha _5$ & 1 & 1 & 0  \\
$\alpha _6$ & 1 & 0 & 1  \\
$\alpha _7$ & 0 & 1 & 1  \\
$\alpha _8$ & 1 & 1 & 1 
\end{tabular}}\\[0.4cm]
\end{tabular}
& &
{\begin{tabular}{l | lll}
     & $f_2$ & $f_3$ & $f_4$   \\[0.01cm] \hline
     &  &  &   \\[-0.351cm]
${\mathbf F^J}(\alpha _1)=\alpha _1$ & 0 & 0 & 0  \\
${\mathbf F^J}(\alpha _2)=\alpha _3$ & 0 & 1 & 0  \\
${\mathbf F^J}(\alpha _3)=\alpha _2$ & 1 & 0 & 0  \\
${\mathbf F^J}(\alpha _4)=\alpha _4$ & 0 & 0 & 1  \\
${\mathbf F^J}(\alpha _5)=\alpha _5$ & 1 & 1 & 0  \\
${\mathbf F^J}(\alpha _6)=\alpha _7$ & 0 & 1 & 1  \\
${\mathbf F^J}(\alpha _7)=\alpha _6$ & 1 & 0 & 1  \\
${\mathbf F^J}(\alpha _8)=\alpha _8$ & 1 & 1 & 1 
\end{tabular}}\\[0.4cm]
& &\\
\end{tabular}\\
\caption{${\mathbf L}^{\mathbf J}$ and  the induced tense operator 
${\mathbf F^{\mathbf J}}$ on ${\mathbf L^{\mathbf J}}$}
\label{tab:label}
\end{table}
We will now describe the relation $S_{{\mathbf J}[{\mathbf L}^{\mathbf J}, \mathbf L]}$, which we denote as $\rho '$. Recall that by definition, for two sup-semilattice homomorphisms 
$\varphi,\psi\colon {\mathbf L^T}\to {\mathbf L}$ we have 
$$\varphi\mathrel{\rho'}\psi\iff(\forall k\in \{1,2,3,4,5,6,7,8\})(\psi(\alpha _k)\leq \varphi({\mathbf F^J}(\alpha _k))).$$

For $\nu_{\mathbf J}\colon \mathbf J\to {\mathbf J}[{\mathbf L}^{\mathbf J}, \mathbf L]$ 
this by definition of $\nu_{\mathbf J}$ means that $\nu_{\mathbf J}(f_i)\rho '\nu_{\mathbf J}(f_j)$ if and only if 
$$\alpha _k(f_j)=(\nu_{\mathbf J}(f_j))(\alpha _k)\leq (\nu_{\mathbf J}(f_i))({\mathbf F^J}(\alpha _k))=({\mathbf F^J}(\alpha _k))(f_i),$$

for all $k\in \{1,2,3,4,5,6,7,8\}$. This allows us to easily describe whether the respective pairs 
$(\nu_{\mathbf J}(f_i),\nu_{\mathbf J}(f_k))$, $i,k \in \{2,3,4\}$ belong to $\rho '$. 
We see that $(\nu_{\mathbf J}(f_2),\nu_{\mathbf J}(f_2))$ and $(\nu_{\mathbf J}(f_3),\nu_{\mathbf J}(f_3))$ don't belong to $\rho '$, since $\alpha _2(f_2)=1>0=({\mathbf F^J}(\alpha _2))(f_2)$ and $\alpha _3(f_3)=1>0=({\mathbf F^J}(\alpha _3))(f_3)$ respectively. Similarly, we can see that $(\nu_{\mathbf J}(f_3),\nu_{\mathbf J}(f_4))$ 
doesn't belong to $\rho '$, since $\alpha _3(f_3)=1>0=({\mathbf F^J}(\alpha _3))(f_3)$ and it can be 
analogically shown that $(\nu_{\mathbf J}(f_4),\nu_{\mathbf J}(f_3))$ doesn't belong to $\rho '$ either. 

Moreover, since  $({\mathbf F^J}(\alpha _k))(f_2)=\alpha _k(f_3)$, $({\mathbf F^J}(\alpha _k))(f_3)=\alpha _k(f_2)$ and $({\mathbf F^J}(\alpha _k))(f_4)=\alpha _k(f_4)$ for all $k\in \{1,2,3,4,5,6,7,8\}$, we get 
$(\nu_{\mathbf J}(f_2),\nu_{\mathbf J}(f_3))\in \rho ', (\nu_{\mathbf J}(f_3), \nu_{\mathbf J}(f_2))\in \rho '$ and 
$(\nu_{\mathbf J}(f_4),\nu_{\mathbf J}(f_4))\in \rho '$. In conclusion, this means $f_i \mathrel{S} f_k$ 
in our original frame ${\mathbf J}$ if and only if $\nu_{\mathbf J}(f_i)\mathrel{\rho '}\nu_{\mathbf J}(f_k)$ in the frame ${\mathbf J}[{\mathbf L}^{\mathbf J}, \mathbf L]$.
\end{example}

\begin{example}\label{thirdexample}
Let ${\mathbf H}=({\mathbf G},F)$ be the $F$-sup-semilattice from Example \ref{firstexample} where ${\mathbf G}=(\{0,a,b,c,1\},\bigvee)$ and $0<a,b,c< 1$. Operator $F$ is given by the prescription  $F(0)=0$, $F(a)=a$, $F(b)=c$, $F(c)=b$, $F(1)=1$. Let  ${\mathbf J}=(T,S)$ be the  frame   from Example  \ref{secondexample} where 
$T=\{f_2,f_3,f_4\}$ and  $S$ is a relation on $\{f_2,f_3,f_4\}$ given by 
$f_2 \mathrel{S} f_3$, $f_3\mathrel{S} f_2$ and   $f_4\mathrel{S} f_4$.

From Theorem \ref{adjSFSL} we obtain a lax morphism 
$\eta _{\mathbf H}\colon \mathbf H\to (\mathbf J\otimes \mathbf H)^{\mathbf J}$ 
of $F$-sup-semilattices defined in such a way that 
$$(\eta _{\mathbf H}(x))(i)=\mbox{\rm n}(j[\mathbf J,\mathbf H])(x_{i=}).$$ 

Let us first construct $\mathbf J\otimes \mathbf H$. To do that, we will need 
$$
[\mathbf J, \mathbf H]=\{(x_{iS}\vee F(x)_{i=}, F(x)_{i=})\mid x\in G, i\in T\}.
$$
By definition, $x_{f_iS}$, $F(x)_{f_i=}$ and $x_{f_iS}\lor F(x)_{f_i=}$ for 
an arbitrary $x\in {\mathbf G}$ are given by the following tables:

\noindent{}%
\begin{tabular}{@{}c@{\,\,\,}c@{}c@{\,\,\,}c@{}c}
{\begin{tabular}{l | lll}
     & $f_2$ & $f_3$ & $f_4$   \\[0.01cm] \hline
     $$ &  &  &       \\[-0.351cm]
$0_{f_2S}$ & $0$ & $0$ & $0$  \\
$0_{f_3S}$ & $0$ & $0$ & $0$  \\
$0_{f_4S}$ & $0$ & $0$ & $0$  \\
$a_{f_2S}$ & $0$ & $a$ & $0$  \\
$a_{f_3S}$ & $a$ & $0$ & $0$  \\
$a_{f_4S}$ & $0$ & $0$ & $a$  \\
$b_{f_2S}$ & $0$ & $b$ & $0$  \\
$b_{f_3S}$ & $b$ & $0$ & $0$  \\
$b_{f_4S}$ & $0$ & $0$ & $b$  \\
$c_{f_2S}$ & $0$ & $c$ & $0$  \\
$c_{f_3S}$ & $c$ & $0$ & $0$  \\
$c_{f_4S}$ & $0$ & $0$ & $c$  \\
$1_{f_2S}$ & $0$ & $1$ & $0$  \\
$1_{f_3S}$ & $1$ & $0$ & $0$  \\
$1_{f_4S}$ & $0$ & $0$ & $1$  \\
\end{tabular}}& &%
{\begin{tabular}{l | lll}
     & $f_2$ & $f_3$ & $f_4$   \\[0.01cm] \hline
     $$ &  &  &     \\[-0.351cm]
$F(0)_{f_2=}$ & $0$ & $0$ & $0$  \\
$F(0)_{f_3=}$ & $0$ & $0$ & $0$  \\
$F(0)_{f_4=}$ & $0$ & $0$ & $0$  \\
$F(a)_{f_2=}$ & $a$ & $0$ & $0$  \\
$F(a)_{f_3=}$ & $0$ & $a$ & $0$  \\
$F(a)_{f_4=}$ & $0$ & $0$ & $a$  \\
$F(b)_{f_2=}$ & $c$ & $0$ & $0$  \\
$F(b)_{f_3=}$ & $0$ & $c$ & $0$  \\
$F(b)_{f_4=}$ & $0$ & $0$ & $c$  \\
$F(c)_{f_2=}$ & $b$ & $0$ & $0$  \\
$F(c)_{f_3=}$ & $0$ & $b$ & $0$  \\
$F(c)_{f_4=}$ & $0$ & $0$ & $b$  \\
$F(1)_{f_2=}$ & $1$ & $0$ & $0$  \\
$F(1)_{f_3=}$ & $0$ & $1$ & $0$  \\
$F(1)_{f_4=}$ & $0$ & $0$ & $1$  \\
\end{tabular}}& &
{\begin{tabular}{l | lll}
     & $f_2$ & $f_3$ & $f_4$   \\[0.01cm] \hline
     $$ &  &  &  \\[-0.351cm]
$0_{f_2S}\lor F(0)_{f_2=}$ & $0$ & $0$ & $0$  \\
$0_{f_3S}\lor F(0)_{f_3=}$ & $0$ & $0$ & $0$  \\
$0_{f_4S}\lor F(0)_{f_4=}$ & $0$ & $0$ & $0$  \\
$a_{f_2S}\lor F(a)_{f_2=}$ & $a$ & $a$ & $0$  \\
$a_{f_3S}\lor F(a)_{f_3=}$ & $a$ & $a$ & $0$  \\
$a_{f_4S}\lor F(a)_{f_4=}$ & $0$ & $0$ & $a$  \\
$b_{f_2S}\lor F(b)_{f_2=}$ & $c$ & $b$ & $0$  \\
$b_{f_3S}\lor F(b)_{f_3=}$ & $b$ & $c$ & $0$  \\
$b_{f_4S}\lor F(b)_{f_4=}$ & $0$ & $0$ & $1$  \\
$c_{f_2S}\lor F(c)_{f_2=}$ & $b$ & $c$ & $0$  \\
$c_{f_3S}\lor F(c)_{f_3=}$ & $c$ & $b$ & $0$  \\
$c_{f_4S}\lor F(c)_{f_4=}$ & $0$ & $0$ & $1$  \\
$1_{f_2S}\lor F(1)_{f_2=}$ & $1$ & $1$ & $0$  \\
$1_{f_3S}\lor F(1)_{f_3=}$ & $1$ & $1$ & $0$  \\
$1_{f_4S}\lor F(1)_{f_4=}$ & $0$ & $0$ & $1$  \\
\end{tabular}}

\end{tabular}

One can easily obtain that $\mathbf J\otimes \mathbf H$ has 15 elements and 
$(\mathbf J\otimes \mathbf H)^{\mathbf J}$ has $15^{3}$ elements. 

Also, we can compute $\eta _{\mathbf H}$ as follows:

\begin{center}
{\begin{tabular}{l | lll}
		& $f_2$ & $f_3$ & $f_4$   \\[0.01cm] \hline
		$$ &  &  &    \\[-0.351cm]
		$\eta _{\mathbf H}(0)$ & $[(0,0,0)]$ & $[(0,0,0)]$ & $[(0,0,0)]$  \\[0.051cm]
		$\eta _{\mathbf H}(a)$ & $[(a,a,0)]$ & $[(a,a,0)]$ & $[(0,0,a)]$  \\[0.051cm]
		$\eta _{\mathbf H}(b)$ & $[(c,b,0)]$ & $[(b,c,0)]$ & $[(0,0,1)]$  \\[0.051cm]
		$\eta _{\mathbf H}(c)$ & $[(b,c,0)]$ & $[(c,b,0)]$ & $[(0,0,1)]$  \\[0.051cm]
		$\eta _{\mathbf H}(1)$ & $[(1,1,0)]$ & $[(1,1,0)]$ & $[(0,0,1)]$  \\
\end{tabular}}
\end{center}
This notation means that when $(\eta _{\mathbf H}(x))(f_i)\in \mathbf J\otimes \mathbf H$ is mapped to $[(\alpha,\beta,\gamma)]$ then $(\eta _{\mathbf H}(x))(f_i)$ is a class of maps $\mathbf J\to \mathbf H$ represented by the map of the form $f_2\mapsto \alpha$, $f_3\mapsto \beta$, $f_4\mapsto \gamma$.

Now we will show that the lax morphism $\eta _{\mathbf H}$ is a homomorphism of $F$-sup-semilattices. To do that, we will have to check that $(\eta _{\mathbf H}(F(x)))$ and $F^\mathbf{J}(\eta _{\mathbf H}(x))$, where $F^\mathbf{J}$ is the induced tense operator $F^\mathbf{J}$ on $(\mathbf J\otimes \mathbf H)^{\mathbf J}$, is the same map for every $x\in G$.

By the definition of $F$ it is clear that $\eta _{\mathbf H}(F(x))$ where $x\in G$ is given by the following table:

\begin{center}
{\begin{tabular}{l | lll}
		& $f_2$ & $f_3$ & $f_4$   \\[0.01cm] \hline
		$$ &  &  &   \\[-0.351cm]
		$\eta _{\mathbf H}(F(0))$ & $[(0,0,0)]$ & $[(0,0,0)]$ & $[(0,0,0)]$  \\[0.051cm]
		$\eta _{\mathbf H}(F(a))$ & $[(a,a,0)]$ & $[(a,a,0)]$ & $[(0,0,a)]$  \\[0.051cm]
		$\eta _{\mathbf H}(F(b))$ & $[(b,c,0)]$ & $[(c,b,0)]$ & $[(0,0,1)]$  \\[0.051cm]
		$\eta _{\mathbf H}(F(c))$ & $[(c,b,0)]$ & $[(b,c,0)]$ & $[(0,0,1)]$  \\[0.051cm]
		$\eta _{\mathbf H}(F(1))$ & $[(1,1,0)]$ & $[(1,1,0)]$ & $[(0,0,1)]$  \\
\end{tabular}}
\end{center}

Let us now describe the induced tense operator $F^\mathbf{J}$. Let $\varphi \in (\mathbf J\otimes \mathbf H)^{\mathbf J}$ be an arbitrary map. By definition of relation $\rho$, it is easy to see that for $F^\mathbf{J}$ the following holds

\[ (F^\mathbf{J}(\varphi))(f_i) =
\begin{cases}
	\varphi (f_3)       & \quad \text{if }  i=2,\\
	\varphi (f_2)       & \quad \text{if }  i=3,\\
	\varphi (f_4)       & \quad \text{if }  i=4.
\end{cases}
\]

This allows us to compute $F^\mathbf{J}(\eta _{\mathbf H}(x))$ where $x\in G$ and compare with the previous table. The respective values are given by the following table:

\begin{center}
{\begin{tabular}{l | lll}
		& $f_2$ & $f_3$ & $f_4$   \\[0.01cm] \hline
		$$ &  &  &    \\[-0.351cm]
		$F^\mathbf{J}(\eta _{\mathbf H}(0))$ & $[(0,0,0)]$ & $[(0,0,0)]$ & $[(0,0,0)]$  \\
		$F^\mathbf{J}(\eta _{\mathbf H}(a))$ & $[(a,a,0)]$ & $[(a,a,0)]$ & $[(0,0,a)]$  \\
		$F^\mathbf{J}(\eta _{\mathbf H}(b))$ & $[(b,c,0)]$ & $[(c,b,0)]$ & $[(0,0,1)]$  \\
		$F^\mathbf{J}(\eta _{\mathbf H}(c))$ & $[(c,b,0)]$ & $[(b,c,0)]$ & $[(0,0,1)]$  \\
		$F^\mathbf{J}(\eta _{\mathbf H}(1))$ & $[(1,1,0)]$ & $[(1,1,0)]$ & $[(0,0,1)]$  \\
\end{tabular}}
\end{center}

We conclude that  $\eta _{\mathbf H}$ is a homomorphism of $F$-sup-semilattices. 
Since $\eta _{\mathbf H}$ is evidently injective we obtain as in Example \ref{firstexample} 
that  $\eta _{\mathbf H}$ is an element of $\mathcal{E}_\leq$.
\end{example}

\section{Conclusions}\label{Conclusions}

In this paper, we have presented three basic construction methods to construct 
\begin{enumerate}[(i)]
	\item an $F$-sup-semilattice from a sup-semilattice and a relation, 
	\item a sup-semilattice from an $F$-sup-semilattice and a relation, and
	\item a relation from an $F$-sup-semilattice  and  a sup-semilattice,
	\end{enumerate}
and obtained three induced adjoint situations betweee the respective categories. 
 This result gives us an unifying view on recent results about representations of tense operators in different categories of posets and lattices.
 
 For future work, there are a number of open problems that we plan to address.
In particular we can mention the following ones: 

\begin{enumerate}
	\item As we have seen in Section \ref{approaches} the lax morphisms 
	$\mu_{\mathbf H}$ from Example \ref{firstexample} and 
	$\eta_{\mathbf H}$ from Example 
	\ref{thirdexample}  were both homomorphisms of $F$-sup-semilattices. We will look 
	for necessary and sufficient conditions to ensure this also in a general case.
	\item We intend to represent or approximate 
	any $\pwo${} $\mathcal A$ in the following manner for a 
	fixed sup-semilattice ${\mathbf L}$. The choice of ${\mathbf L}$ will of course depend on what behaviour we will be looking at (one could choose ${\mathbf L}$ to be a Boolean algebra or $MV$-algebra or simply a finite chain). We plan to construct a suitable completion 
	${\mathbf H}_{\mathcal A}$ of $\mathcal A$ and using (iii) we obtain a new frame 
	${\mathbf J}_{\mathcal A}$. The composition of $\mathcal A\to {\mathbf J}_{\mathcal A}$ 
	and $\eta _{{\mathbf H}_{\mathcal A}}$ will give us the desired representation or approximation. 
	\item  We intend to add more structure to our starting category of  sup-semila\-ttices. 
	More precisely, we plan to substitute this category  by a category of sup-algebras of a given type and also the category of  $F$-sup-semilattices 
	by a category of $F$-sup-algebras of the same type, and try to answer the same questions as above. 
\end{enumerate}





\end{document}